\documentclass{amsart}

\usepackage{amsthm}
\usepackage{url}
\usepackage{amsfonts}
\usepackage{amssymb}
\usepackage{tensor}
\usepackage{bbm}
\usepackage{mathrsfs}
\usepackage{mathtools}
\numberwithin{equation}{section}
\usepackage{color}
\usepackage{hyperref}
\usepackage{mathtools}
\mathtoolsset{showonlyrefs}
\usepackage{graphicx}
\usepackage{caption}
\usepackage{subcaption}
\newtheorem{theorem}{Theorem}
\newtheorem{lemma}[theorem]{Lemma}

\newtheorem{remark}[theorem]{Remark}
\newtheorem{example}[theorem]{Example}

\newcommand{\norm}[2]{\ensuremath{\|#2\|_{#1}}}

\newcommand{\abs}[1]{\ensuremath{|#1|}}

\newcommand{\aabs}[2]{\ensuremath{|#2|_{#1}}}
\newcommand{\nnnorm}[2]{{\left\vert\kern-0.25ex\left\vert\kern-0.25ex\left\vert 
#2 
    \right\vert\kern-0.25ex\right\vert\kern-0.25ex\right\vert_{#1}}}

\newcommand{\nnorm}[2]{{\left\vert\kern-0.25ex\left\vert\kern-0.25ex\left\vert 
#2 
    \right\vert\kern-0.25ex\right\vert\kern-0.25ex\right\vert^{2}_{#1}}}

\newcommand{\inner}[2]{\ensuremath{\langle #1,#2 \rangle}_{H}}

\newcommand{\inners}[3]{\ensuremath{\langle #1,#2 \rangle}_{#3}}
\newcommand{\duals}[4]{\ensuremath { \tensor[_{#3}]{\langle#1,#2\rangle 
}{_{#4}}  }}

\newcommand{\dual}[2]{\ensuremath { \tensor[_{V^*}]{ \langle#1,#2\rangle }{_V}  
  }}
\newcommand{\tdual}[2]{\ensuremath { \tensor[_{V}]{ \langle#1,#2\rangle
}{_{V^*}} }}

\newcommand{\bR}{{\mathbb R}}

\newcommand{\bP}{{\mathbb P}}
\newcommand{\bE}{{\mathbb E}}

\newcommand{\dd}{{\mathrm{d}}}

\newcommand{\cY}{\mathcal{Y}}

\newcommand{\ctiiXo}{\ensuremath {\mathcal{X}}_{\omega}^{h,k,q+1}}

\newcommand{\ctiiX}{\ensuremath {\mathcal{X}}^{h,k,q+1}}
\newcommand{\ctiiY}{\ensuremath {\mathcal{Y}}^{h,k,q}}

\newcommand{\cX}{\mathcal{X}}

\newcommand{\cL}{\mathcal{L}}

\newcommand{\cT}{\mathcal{T}}

\newcommand{\AM}{\ensuremath A_{\max}}
\newcommand{\Am}{\ensuremath A_{\min}}
\newcommand{\Ao}{\ensuremath A_{\omega}}

\newcommand{\RM}{\ensuremath {\rho}_{\max}}

\newcommand{\sC}{\mathscr{C}}

\newcommand{\sL}{\mathscr{L}}

\newcommand{\sB}{\mathscr{B}}

\newcommand{\sF}{\mathscr{F}}

\newcommand{\til}[1]{{\tilde #1}}

\title[Quasi-optimality in parabolic 
problems with random coefficients]{Quasi-optimality of 
Petrov-Galerkin discretizations of parabolic 
problems with random coefficients}

\author[S.~Larsson]{Stig Larsson}

\address{
  Department of Mathematical Sciences,
  Chalmers University of Technology and University of Gothenburg, 
  SE--412 96 Gothenburg,
  Sweden}

\email{stig@chalmers.se}

\author[C.~Mollet]{Christian Mollet}

\address{
  Department of Mathematics,
  University of Cologne,
  DE--50931 K\"oln,
  Germany}

\email{cmollet@math.uni-koeln.de}

\author[M.~Molteni]{Matteo Molteni}

\address{
  Department of Mathematical Sciences,
  Chalmers University of Technology and University of Gothenburg,
  SE--412 96 Gothenburg, 
  Sweden}

\email{molteni@chalmers.se}



\begin{document}

\begin{abstract}
We consider a linear parabolic problem with random elliptic operator in the 
usual
Gelfand triple setting. We do not assume uniform bounds on the coercivity and 
boundedness constants, but allow them to be random variables. The parabolic 
problem is studied in a weak space-time formulation, where we can derive 
explicit formulas for the inf-sup constants. Under suitable assumptions we prove 
existence of moments of the solution. We also prove quasi-optimal error 
estimates for piecewise polynomial Petrov-Galerkin discretizations.
\end{abstract}

\date{\today}

\maketitle
 
\section{Introduction} 
Let $T \in
(0,\infty)$ be fixed and let $(\Omega,\Sigma,\bP)$ be a complete
probability space, with normal filtration $\Sigma = (\Sigma_t)_{t \in [0,T]}$. 
We consider a linear parabolic evolution problem with random coefficients of the
form
\begin{equation} \label{eq:randHeatStrong}
\begin{aligned}
&\dot{u}(t,\omega) + A(t,\omega)u(t,\omega) = f(t,\omega) &&\mbox{in
$V^*$,}\,\, t\in(0,T],\\
&u(0,\omega) = u_0(\omega)&&\mbox{in } H,
\end{aligned}
\end{equation}
where the equations are understood $\bP$-a.s., with respect to $\omega \in 
\Omega$. We assume that $A$ is a 
random elliptic operator defined within a
Gelfand triple $V\subset H\subset V^*$, where $V$ is
continuously and densely embedded into $H$, and both are separable 
Hilbert spaces. We denote by $
\inner{\cdot}{\cdot}$ the inner product in $H$ and by
$\dual{\cdot}{\cdot}$ the dual pairing between $V$ and $V^*$, where $
\inner{u}{v} = \dual{u}{v}$ when $u \in H, v\in V$.  We assume that a 
progressively measurable map $ A \colon [0,T] 
\times \Omega \to 
\sL(V,V^*)$, is given, and that it is coercive and bounded uniformly in $[0,T]$ 
but not necessarily in $\Omega$. We denote by $a(\cdot,\cdot)$ its associated 
bilinear form,
given by
$a(t,\omega;u,v) := \dual{A(t,\omega)u}{v}$. Finally, we assume that $f$ is a 
progressively
measurable process with Bochner integrable trajectories, that the initial 
data $u_0$ is measurable, and that for fixed $\omega$ they belong to 
$L^2((0,T);V^*)$, respectively to $H$.

In \cite{LarssonMolteni} a new formulation for the linear stochastic
heat equation driven by additive noise was introduced, based on a
space-time variational formulation for its deterministic
counterpart. The choice of such a formulation is motivated by the lack
of regularity of the solution of the stochastic heat equation.  By
using this formulation $\omega$-wise it was possible to produce a
solution that exists almost surely and, with the further assumption
that the bounds on $A$ are uniform in $\Omega$, to show that the
solution is square integrable over $\Omega$.

The lack of noise in \eqref{eq:randHeatStrong} represents of course a
simplification for the application of the theory developed in
\cite{LarssonMolteni}, so that we can obtain in the same way existence,
uniqueness and a bound for the solution in terms of the data $u_0$ and $f$. The
solution obtained is then Bochner square
integrable over $\Omega$ and, in particular, the stochastic components of the 
problem can be naturally incorporated in the Hilbert spaces over which the 
problem is formulated, thus producing a weak stochastic-space-time formulation. 
However this
approach presents the drawback of having restrictive hypothesis
on $A$, namely uniform bounds $\bP$-a.s., both from below and from above, which 
are not strictly necessary.

This suggests the possibility of trying another kind of approach in which 
the bounds for
$A$ are not necessarily
uniform in $\omega$, but rather random variables, finite
and positive $\bP$-a.s. The solution
is then shown to belong to $L^p$ with respect to $\Omega$ only afterwards, 
when results of existence and 
uniqueness are already obtained $\omega$-wise. We prove a sufficient condition 
for the 
existence of 
$p$-moments of the solution in terms of moments of the random variables 
bounding $A$. The results are first 
presented under the general assumption of 
having a non-self-adjoint, time-dependent operator, and then restricted  
to operators that are self-adjoint and to operators that are 
also constant in time.

The motivation for doing this, comes from the perspective of studying the
application of multilevel Monte Carlo methods to \eqref{eq:randHeatStrong}. The 
estimates we provide are part of what is
required in order to bound the cost of the multilevel Monte Carlo estimator,
for a given tolerance. In this respect, keeping track of the constants arising
from the error estimates, plays a crucial role.

A similar analysis, for an elliptic partial differential equation with random
coefficients, can be found in \cite{Aretha, C12}, where the authors establish
bounds on the discretization error in the relevant case of non-uniformly
coercive or bounded coefficients, with respect to the random parameter. Our
goal is to provide a first step in the same direction for developing an
analogous theory in the parabolic case.

The article is structured as follows. We first state in 
Section~\ref{sec:preliminary1} the abstract form of the main theorem upon which 
we rely for proving existence and uniqueness. 
In Section~\ref{sec:formAS} we state the weak
space-time formulation and show that an almost sure solution to
\eqref{eq:randHeatStrong} exists; this is done in three stages, by 
progressively assuming more on the operator $A$ we obtain  
more detailed estimates. In Section~\ref{sec:quasi_optimality} 
we use the results from the previous section in order to derive results 
of quasi-optimality for both a spatial semidiscretization and a full discretization 
of the original problem. In Section~\ref{sec:Lpest} we prove that 
if the input data has finite moments of certain orders, then the solution also has finite
moments of some (typically) lower order. 
Moreover we show that for the semidiscrete scheme we have a 
quasi-optimal constant that depends on $\omega$ only through the
ratio between the boundedness and coercivity constants of $A$, while for the 
fully-discrete scheme we have an additional factor depending on the
boundedness constant. In both cases we extend the $\omega$-wise results of 
quasi-optimality to $L^p$ with respect to $\Omega$. Numerical experiments that support our 
theoretical findings are finally presented in Section~\ref{sec:num_res}.


\section{The inf-sup theory}\label{sec:preliminary1}
The main tool upon which we base our results is 
the Banach--Ne{\v{c}}as--Babu{\v{s}}ka (BNB) theorem (see, e.g., \cite{NSV}).
\begin{theorem}\label{thm:BNB}
 Let $V$ and
$W$ be Hilbert spaces, and consider a bounded bilinear
form $\sB \colon W \times V \rightarrow \bR $, with 
\begin{equation}
\tag{BDD} \label{eq:BDD}
C_B : =  \sup_{0 \neq w \in W}\sup_{0 \neq v \in
V}\frac{\sB(w,v)}{\norm{W}{w}\norm{V}{v}} < \infty,\\
\end{equation}
and the associated bounded linear operator $B \colon W \rightarrow V^*$, i.e.,
$B\in\sL(W,V^*)$, defined by 
\begin{align*}
\duals{Bw}{v}{V^*}{V} := \sB(w,v) , \quad \forall v 
\in V, \, w
\in W.
\end{align*}
The operator $B$ is boundedly invertible if and only
if the following conditions are satisfied:
\begin{align}
\tag{BNB1} \label{eq:BNB1A}
&c_B :=  \inf_{0 \neq w \in W}\sup_{0 \neq v \in
V}\frac{\sB(w,v)}{\norm{W}{w}\norm{V}{v}} >0,\\
\tag{BNB2} \label{eq:BNB2A}
&\forall v \in V, \quad \sup_{0 \neq w \in W}\sB(w,v) >0.
\end{align} 
\end{theorem}

The constant $c_B$ is called the inf-sup constant and, whenever 
\eqref{eq:BNB1A}--\eqref{eq:BNB2A} 
hold, we have the equivalent condition
\begin{align}\label{eq:interchangespaces}
\inf_{0 \neq w \in W}\sup_{0 \neq v \in
V}\frac{\sB(w,v)}{\norm{W}{w}\norm{V}{v}} = \inf_{0 \neq v \in V}\sup_{0 \neq w
\in W}\frac{\sB(w,v)}{\norm{W}{w}\norm{V}{v}} >0,
\end{align}
which allows to swap the spaces where the infimum and the supremum are taken, 
see, e.g., \cite[Th. 2.2]{NSV}.

An immediate consequence of this is that the variational problem
\begin{equation*}
w \in W \colon \sB(w,v) = F(v), \quad \forall v\in V, \quad F \in V^*,
\end{equation*}
i.e., $Bw=F$ in $V^*$, and its adjoint 
\begin{equation*}
v \in V\colon \sB(w,v) = G(w), \quad \forall w\in W, \quad G \in W^*,
\end{equation*}
i.e., $B^*v=G$ in $W^*$, are well-posed whenever \eqref{eq:BDD}, 
\eqref{eq:BNB1A} and
\eqref{eq:BNB2A} hold. In particular, the well-posedness of the former is
equivalent to the well-posedness of the latter and the norms of the solutions
are bounded, respectively, by
\begin{align*}
\norm{W}{w} \leq \frac{1}{c_B}\norm{V^*}{F}, \quad \norm{V}{v} \leq
\frac{1}{c_B}\norm{W^*}{G}.
\end{align*}
The strength of this theorem is the fact that the bounding constants for 
the norm of the 
solution are known explicitly. In particular, by combining Theorem~\ref{thm:BNB} 
with the results of quasi-optimality presented later in this 
article, we will be able to keep track of every constant arising when 
estimating the error of the numerical solution.


\section{An $\omega$-wise weak space-time formulation}\label{sec:formAS}
The approach that we want to introduce relies on the one first introduced in 
\cite{LarssonMolteni} for a stochastic evolution problem. We derive the 
weak formulation in the same way as in
the deterministic case, obtaining a family of problems depending on $\omega
\in \Omega$, which can be solved $\bP$-a.s.

In this case, the problem reads exactly as its deterministic counterpart, for
which many works have been produced during last years. We refer in particular 
to
\cite{BabuskaJanik,SchwabSte,Fra},
for what is known as the {\it first formulation}, and to
\cite{CheginiStev,LarssonMolteni,Mollet,SchwabSuli,Fra,UrbanPatera},
for what is known as the {\it second (or weak) 
formulation}, which is the one that we use in the present work.

Starting from the abstract parabolic equation given in the Gelfand
triple framework presented above, we assume now that the bilinear form 
satisfies the following conditions for some 
{\rm random
variables} $\Am=\Am(\omega)$ and $\AM=\AM(\omega)$, positive and finite, 
$\bP$-a.s.:
\begin{alignat}{4}
& \abs{a(t,\omega;u,v)} \leq \AM(\omega) \norm{V}{u}
\norm{V}{v},\quad
&& t
\in [0,T],&& \, u,v\in V, \\
& a(t,\omega;v,v) \geq \Am(\omega) \norm{V}{v}^2, \quad 
&& t \in
[0,T],
&&\, v\in V.
\end{alignat}
We make use 
of the following Lebesgue-Bochner spaces
\begin{align*}
\cY &:= L^2((0,T);V), \\
\cX &:= L^2((0,T);V) \cap H^1((0,T);V^*),  &&\cX_{0,\{T\}}:=\{x\in\cX:x(T)=0\},
\end{align*}
normed by
\begin{align*}
\norm{\cY}{y}^2 &:= \norm{L^2((0,T);V)}{y}^2, \\
\norm{\cX}{x}^2 &:= \norm{H}{x(0)}^2 + \norm{L^2((0,T);V)}{x}^2 + 
\norm{L^2((0,T);V^*)}{\dot{x}}^2.
\end{align*}
The space $\cX_{0,\{T\}}$ is densely embedded in $ \sC([0,T];H)$, 
(\cite[Theorem 
1, 
Chapter 
XVIII.1]{DauLio}), meaning that pointwise values of $x$ are well defined, in 
particular, $x(0)$ in \eqref{eq:loadfunc}. 
We define a parametrized bilinear form, given by 
\begin{equation}\label{eq:bilinear_form}
\begin{aligned}
&\sB_{\omega}^* \colon \cY \times \cX_{0,\{T\}} \rightarrow \bR,\\
&\sB_{\omega}^*(y,x) := \int_0^T{\Big( \tdual{y(t)}{-\dot{x}(t)} 
+a(t,\omega;y(t),x(t))
\Big)}\,\dd t,
\end{aligned}
\end{equation}
and a parametrized load functional, given by
\begin{equation}\label{eq:loadfunc}
\begin{aligned}
&\sF_{\omega} \colon \cX_{0,\{T\}} \rightarrow \bR,\\
&\sF_{\omega}(x) := \int_0^T {\dual{f(t,\omega)}{x(t)}}\,\dd 
t +
\inner{u_0(\omega)}{x(0)}.
\end{aligned}
\end{equation}

With this notation, we obtain that
the weak space-time formulation of the original problem 
\eqref{eq:randHeatStrong} is:
\begin{equation} \label{eq:secondspacetime}
 u = u(\cdot,\omega) \in \cY: \sB_{\omega}^*(u,x) = \sF_{\omega}(x), \quad 
\forall x \in
\cX_{0,\{T\}}, \quad {\rm a.e. } ~ \omega \in \Omega.
\end{equation}
We present three different proofs of the validity of the conditions expressed in 
Theorem~\ref{thm:BNB}; by progressively assuming more on the operator $A$, we 
provide sharper estimates for the norm of the solution and bounds in more 
general norms.

\subsection{The inf-sup theorem: general case}\label{subsec:inf_sup_gen}

The inf-sup constant, the value of which is not very relevant in
the deterministic framework, has an important role in our work. Indeed, since 
we do no 
longer assume
uniform boundedness in $\Omega$, but only almost sure finiteness, the existence
of $p$-moments for the solution to \eqref{eq:secondspacetime} depends on the
existence of higher moments for $\AM$, $\Am^{-1}=\frac{1}{\Am}$, $u_0$ and $f$. 
Therefore having the sharpest possible bound is
crucial. We introduce the following notation:
\begin{align*}
\rho(\omega) := \frac{\AM(\omega)}{\Am(\omega)}, 
\end{align*}
the role of which will be more evident in 
Sections~\ref{sec:quasi_optimality}--\ref{sec:Lpest}.


Typical estimates for $c_B$ and $C_B$ appearing in Theorem 
\ref{thm:BNB}, adapted to our framework, are given by the following theorem 
(see \cite{SchwabSte} and \cite{Fra}):
\begin{theorem}\label{thm:bnbas}
The bilinear form defined in \eqref{eq:bilinear_form}
satisfies the following:
\begin{align}
&C_B(\omega):=\sup_{0\neq y \in \cY }\sup_{0\neq x \in \cX_{0,\{T\}}} 
\frac{\sB_{\omega}^*(y,x)}{ \norm{\cY}{y} \norm{\cX}{x} } \leq 
\sqrt{2}\max\{1,\AM(\omega)\}, \quad 
\bP\mbox{\rm{-a.s.}}, 
\label{eq:bdd1}\\
&c_B(\omega) :=\inf_{0\neq y \in \cY }\sup_{0\neq x \in \cX_{0,\{T\}}} 
\frac{\sB_{\omega}^*(y,x)}{ \norm{\cY}{y} \norm{\cX}{x}} \geq 
\frac{\min\{\Am(\omega), 
\rho^{-1}(\omega) \}}{\sqrt{2}}, \quad 
\bP\mbox{\rm{-a.s.}},\label{eq:infsup1}\\
&\forall x\in \cX_{0,\{T\}} \qquad \sup_{0\neq y \in 
\cY} \sB_{\omega}^*(y,x)>0, \quad 
\bP\mbox{\rm{-a.s.}}\label{eq:infsup12}
\end{align}
\end{theorem}
This shows that the operator $B_{\omega} \in \cL(\cY,\cX_{0,\{T\}}^*)$ 
associated 
with $
\sB_{\omega}^*(\cdot,\cdot)$ via 
\begin{equation}\label{eq:DefB}
 \sB_{\omega}^*(y,x) = \duals{B_{\omega}y}{x}{\cX_{0,\{T\}}^*}{\cX_{0,\{T\}}}
 \end{equation}
is boundedly
invertible $\bP$-a.s. Moreover, for $ f(\cdot,\omega) \in L^2((0,T);V^*)$
and $ u_0(\omega) \in H$, we have $\sF_{\omega} \in \cX_{0,\{T\}}^*$, hence, 
\eqref{eq:secondspacetime} is
well-posed $\bP$-a.s. and admits a unique solution $u(\cdot,\omega) \in \cY$,
such that
\begin{equation*}
\norm{\cY}{u(\cdot,\omega)} \leq c_B^{-1}(\omega)  \norm{\cX^*}{\sF_{\omega}}  
=c_B^{-1} (\omega)
(\norm{\cY^*}{f(\cdot,\omega)}^2 +\norm{H}{u_0(\omega)}^2)^{\frac12}, 
\quad 
\bP\mbox{\rm{-a.s.}},
\end{equation*}
or, more explicitly:
\begin{equation}\label{eq:normboundAS_nt}
\norm{\cY}{u(\cdot,\omega)} \leq \sqrt{2}\, \max\{ \Am^{-1}(\omega), 
\rho(\omega) \} (\norm{\cY^*}{f(\cdot,\omega)}^2 
+
\norm{H}{u_0(\omega)}^2)^{\frac12}, \, \bP\mbox{\rm{-a.s.}}
\end{equation}

This result is however not completely satisfactory, since the 
bound for the norm of the solution expressed in \eqref{eq:normboundAS_nt} 
involves both $\Am$ and $\AM$, where the latter appears as part of the 
quotient defining $\rho$.
Despite the fact that in order to build our error analysis we will 
need to assume that $\rho$ is uniformly bounded, as expressed in 
\eqref{eq:assumption_ratio}, in the 
next two theorems we show that when $A$ is self-adjoint, the 
presence of $\AM$ and $\rho$ can be avoided.

\subsection{The inf-sup theorem: self-adjoint 
case}\label{subsec:inf_sup_alt}
We assume now that $A$ is also self-adjoint, possibly still time-dependent.

The next proof is of particular relevance in this framework, since we 
want to avoid having any unnecessary constant appearing in our estimates. To do 
so, we ``hide'' the constants, for each $\omega$, inside the norms we use. This 
is done in the spirit of the results presented in \cite{LarssonMolteni2}, to 
which we refer for the missing details.

We consider $\omega$ to be fixed, and use the compact notation $A_{\omega}(t)$ 
for the operator 
appearing in \eqref{eq:randHeatStrong}. In virtue of the properties of 
$\Ao(t)$, fractional powers are indeed 
well 
defined and the 
norms of $V$ and $V^*$ are equivalent to $\norm{H}{\Ao^{\frac12}(t) \cdot}$, 
respectively $\norm{H}{\Ao^{-\frac12}(t) \cdot}$.
We therefore introduce 
accordingly two equivalent norms on $\cX_{0,\{T\}}$ and $\cY$ as follows:
\begin{equation}\label{eq:tilde_norms}
\begin{aligned}
\aabs{\cX_{\omega}}{x}^2 &:= \norm{H}{x(0)}^2 + \int_0^T{ \big(\norm{H}{ 
\Ao^{\frac12}(t) x(t)}^2 
+ 
\norm{H}{\Ao^{-\frac12}(t) \dot{x}(t)}^2\big)}\,\dd t,\\
\aabs{\cY_{\omega}}{y}^2 &:=  \int_0^T{ \norm{H}{ \Ao^{\frac12} (t)
y(t)}^2}\,\dd t.
\end{aligned}
\end{equation}

\begin{theorem}\label{thm:bnbthm}
If $A_{\omega}(t)$ is self-adjoint, then the bilinear form in 
\eqref{eq:bilinear_form} is such 
that:
\begin{align}
&C_B:=\sup_{0\neq y \in \cY }\sup_{0\neq x \in \cX_{0,\{T\}}} 
\frac{\sB_{\omega}^*(y,x)}{ \aabs{\cY_{\omega}}{y} \aabs{\cX_{\omega}}{x}} = 
1,\quad 
\bP\mbox{\rm{-a.s.}}, \label{eq:bdd}\\
&c_B:=\inf_{0\neq y \in \cY }\sup_{0\neq x \in \cX_{0,\{T\}}} 
\frac{\sB_{\omega}^*(y,x)}{ \aabs{\cY_{\omega}}{y} \aabs{\cX_{\omega}}{x}} = 
1,\quad 
\bP\mbox{\rm{-a.s.}},\label{eq:infsup}\\
&\forall x\in \cX_{0,\{T\}} \qquad \sup_{0\neq y \in 
\cY} \sB_{\omega}^*(y,x)>0, \quad 
\bP\mbox{\rm{-a.s.}}\label{eq:infsup2}
\end{align}
\end{theorem}

The proof of this is based on \cite[Theorem~3]{LarssonMolteni2}. 
Furthermore, we can estimate the right-hand side in \eqref{eq:loadfunc} as 
follows:
\begin{align}\label{eq:norm_rhs}
\norm{(\cX, \aabs{\cX_{\omega}}{\cdot})^*}{\sF_{\omega}} \leq \Big[ 
\int_0^T{ \norm{H}{\Ao^{-\frac12}(t)f(t,\omega)}^2 
}\,\dd t + \norm{H}{u_0(\omega)}^2 \Big]^{\frac12},\quad 
\bP\mbox{\rm{-a.s.}}
\end{align}
By combining this version of the inf-sup theorem with the estimate in 
\eqref{eq:norm_rhs}, we can thus achieve 
the sharper estimate
\begin{align*}
\int_0^T{ \norm{H}{\Ao^{\frac12}(t) u(t,\omega)}^2 }\,\dd t \leq 
\int_0^T{ \norm{H}{\Ao^{-\frac12}(t) f(t,\omega)}^2 }\,\dd t + 
\norm{H}{u_0(\omega)}^2,\quad 
\bP\mbox{\rm{-a.s.}}
\end{align*}
Hence, by using the 
equivalence between the norms, we obtain that:
\begin{align}\label{eq:bound_norm_sol_V}
\Am(\omega) \int_0^T{ \norm{V}{u(t,\omega)}^2 }\,\dd t \leq  
\Am^{-1}(\omega) \int_0^T{ \norm{V^*}{f(t,\omega)}^2 }\,\dd t + 
\norm{H}{u_0(\omega)}^2,\quad 
\bP\mbox{\rm{-a.s.}}
\end{align}
Hiding the $\omega$-dependence for better readability, we conclude
\begin{equation}\label{eq:normboundAS_nt2}
\norm{ L^2((0,T);V) }{u}^2 \leq  \Am^{-2} \norm{L^2((0,T);V^*) }{f}^2 
+ \Am^{-1} \norm{H}{u_0}^2, \quad \bP\mbox{\rm{-a.s.}}
\end{equation}

The importance of this proof is that the $\omega$ dependence is transferred to 
the spaces on which we solve the problem, so that $C_B$ and $c_B$ are now 
constants and not random variables. This fact will be particularly useful in 
the derivation of results of quasi-optimality, and in Section~\ref{sec:Lpest} 
when integrating over $\Omega$ under relatively weak assumptions
on the data $\AM$, $\Am$, $u_0$ and $f$.

Finally, we notice that the expression in \eqref{eq:normboundAS_nt2} 
is not in the form in which norm bounds derived from the inf-sup theory 
are usually stated, which would be:
\begin{equation}\label{eq:normboundAS_nt3}
\norm{L^2((0,T);V) }{u}^2 \leq  \max\big\{ \Am^{-2}, \Am^{-1} \big\} \big( 
\norm{L^2((0,T);V^*) }{f}^2 
+   \norm{H}{u_0}^2 \big), \quad \bP\mbox{\rm{-a.s.}}
\end{equation}


\subsection{The inf-sup theorem: further spatial 
regularity}\label{subsec:further_spatial}
In the case of an operator $A$ that does not depend on time, it is 
possible to prove further results of spatial regularity. Let $W_+ 
\hookrightarrow  W_- $ 
be Hilbert spaces with scalar products 
$  \inners{\cdot}{\cdot}{W_+} $ and 
$\inners{\cdot}{\cdot}{W_-}$, respectively, and
$W_0:=[W_-,W_+]_{1/2}$ the interpolation space with scalar
product $\inners{\cdot}{\cdot}{W_0}$.
We assume that the operator $A$ is time-independent and self-adjoint with
\begin{equation*}\label{eq:bounds_for_A_shifted}
A_{\rm min} (\omega)\le \|A_\omega\|_{\sL(W_+, W_-)} \le A_{\rm max}(\omega),
\end{equation*}
for random variables $0<\Am\le\AM<\infty$, $\bP${-a.s.}

We introduce the shifted trial and test spaces
\begin{align*}
 \til{\cY}&:=L^2((0,T);W_{-}^*),\\
\til\cX&:=L_2((0,T);W_+)\cap H^1((0,T);W_{-}), \quad
\til\cX_{0,\{T\}}:=\{x\in \til{\cX} : x(T)=0\},
\end{align*}
equipped with norm
\begin{align*}
\|y\|_{\til{\cY}}^2:=\int_0^T \|y(t)\|_{W_-^*}^2\,\dd t,
\end{align*}
as well as the $\omega$-dependent norm
\begin{align*}
|x|_{\til{\cX}_{\omega}}^2:=\int_0^T (\|\dot{x}(t)\|_{W_-}^2 + 
\|A_{\omega}x(t)\|_{W_-}^2)\,\dd t
+ \|A_\omega^{\frac12}x(0)\|_{W_-}^2.
\end{align*}
The following 
proposition gives an inf-sup constant in a more general framework.
\begin{theorem}\label{thm:reg_prop2}
The operator ${B_\omega} \colon {\til{ \cY}} \to {\til{ \cX}_{0,\{T\}}}^*$ 
defined by 
\eqref{eq:DefB} is boundedly invertible with
\begin{equation*} 
\sup_{0 \neq v \in {{\til\cY}} }
\sup_{0 \neq w\in {{\til\cX}_{0,\{T\}}} } 
\frac{|\langle {B_\omega}v , w\rangle|}{\|v\|_{{{\til\cY}}} |w|_{{\til\cX}_\omega}} =
1,  \quad \bP\mbox{\rm{-a.s.}},
\end{equation*}
as well as
\begin{equation}\label{eq:reg_prop2_infsup}
\inf_{0 \neq v \in {{\til\cY}} }
\sup_{0 \neq w\in {{\til\cX}_{0,\{T\}}} } 
\frac{|\langle {B_\omega}v , w\rangle|}{\|v\|_{{\til{\cY}}} |w|_{\til{\cX}_\omega}} = 
1,  \quad \bP\mbox{\rm{-a.s.}}
\end{equation}
\end{theorem}
\begin{remark}
Reasonable choices of $W_-$ and $W_+$ 
for elliptic operators of order $2m$ on a 
bounded domain $\Lambda$ are
\begin{align*}
W_+:=H^{m+\alpha}(\Lambda) \hookrightarrow W_0=H^{\alpha}(\Lambda) 
\hookrightarrow 
W_-:=H^{-m+\alpha}(\Lambda).
\end{align*}
One arrives at the canonical situation when choosing shift parameter 
$\alpha:=0$. 
The setup in \S~\ref{subsec:inf_sup_gen} and \S~\ref{subsec:inf_sup_alt} for 
time-independent operators
is covered by $W_+:=V$ and $W_-:=V^*$ and the choice made in \cite{CheginiStev} 
by
$W_+:=W$ and $W_-:=H$ with the notion from \cite{CheginiStev} including 
explicit 
bounds.
\end{remark}

\begin{proof}
We hide the $\omega$-dependence in some terms for better readability. 
The 
continuity \eqref{eq:BDD} follows by using the Cauchy--Schwarz inequality, boundedness 
of the spatial operator and H\"older's inequality:
\begin{eqnarray*}
|\langle {B_{\omega}} y , x \rangle|
&=& \int_0^T \duals{y(t)}{A_{\omega}x(t)-\dot x(t)}{W_-^*}{W_-} \,\dd t \\
&\le& \int_0^T \|y(t)\|_{W_-^*} \|A_{\omega}x(t)-\dot x(t)\|_{W_-}\, \dd t 
\le  \|y\|_{\til{\cY}} |x|_{\til{\cX}_\omega}, 
\end{eqnarray*}
since
$
-2\int_0^T\langle A_{\omega}x(t),\dot x(t)\rangle_{W_-} \dd t
= \|A_{\omega}^{\frac12}x(0)\|_{W_-}^2.
$
Next we prove the inf-sup condition \eqref{eq:BNB1A}, with interchanged 
trial and test spaces first. 
To this end, we
choose $y_x:=R_{W_-}(A_{\omega}x-\dot x) \in {\til \cY}$, with
Riesz isomorphism $R_{W_-}\colon W_- \to W_-^*$, according to Riesz 
representation theorem
for arbitrary $x \in {\til \cX}_{0,\{T\}}$.
With this choice we obtain
\begin{equation*}
 \duals{{B_{\omega}}y_x}{x}{}{}
 = \int_0^T (\|\dot x(t)\|_{W_-}^2
 + \|A_{\omega} x(t)\|_{W_-}^2)\, \dd t 
 - 2 \int_0^T \inners{ A_{\omega}x(t)}{\dot x(t)}{W_-}\, \dd t
 = |x|_{\til{\cX}_\omega}^2.
\end{equation*}
Combining this with
$ \|y_x\|_{{ \til\cY}}^2 
=
 |x|_{{ \til\cX_{\omega}}}^2 
$ 
proves \eqref{eq:BNB1A} with interchanged spaces.

What remains to show is the surjectivity \eqref{eq:BNB2A} with interchanged 
trial and test spaces.
This part relies on the approach from \cite[Prop. 2.2 and Prop. 
2.3]{Fra} so we only sketch it here. If we assume that there is 
$y^* \in {\til  \cY}\setminus \{0\}$ satisfying
\begin{align*}
\duals{{B_{\omega}}y^*}{x}{}{} = 0 
\qquad \mbox{for all } x \in {\til \cX}_{0,\{T\}},
\end{align*}
then we can conclude that
\begin{equation*}
\int_0^T \duals{y^* (t) }{{\dot x(t) }}{W_-^*}{W_-}\, \dd t
\lesssim \int_0^T \|y^*(t) \|_{W_-^*} \|x(t) \|_{W_+}\,\dd t
< \infty,
\end{equation*}
for all $x \in {\tilde \cX}_{0,\{T\}}$, where we hide the (pathwise) constants 
since 
they 
are not relevant here.
Therefore, it follows that $y^* \in \bar{\cY} := L_2((0,T);W_-^*)\cap 
H^1((0,T);W_+^*)$, by the 
definition of weak derivatives.
Due to the higher regularity, we can integrate by parts and  conclude
$\dot y^*(t)+A_{\omega}y^*(t) =0$ in $W_+^*$ for a.e. $t\in (0,T)$ with 
$y^*(0)=0$. 
The affine transformation $\bar{y}^*(\cdot):=y^*(T-\cdot)$ now yields
\begin{equation}\label{B_bar_def}
-\frac{\dd}{\dd t}\bar{y}^*(t) + A_{\omega}\bar{y}^*(t) = 0, \qquad 
\bar{y}^*(T)=0.
\end{equation}
When we switch the roles $W_-^* \leftrightarrow W_+$ and $W_+^* \leftrightarrow 
W_-$ in the second step of the proof, we obtain
\begin{equation*}
0 = \duals{\bar{B}_{\omega}\bar{y}^*}{z}{{\bar \cX}^*}{{{\bar \cX}}}
 \gtrsim \|\bar{y}^*\| ^2_{\bar \cY},
\end{equation*}
for appropriately chosen $z \in \bar{\cX}:=L_2((0,T),W_+)$ and 
$\bar{B}_{\omega}$ defined via \eqref{B_bar_def}.
To this end,
we can conclude immediately that $y^*= 0$, which 
is a contradiction and therefore yields surjectivity.
\end{proof}
The previous theorem shows that the spatial regularity is inherited to the 
parabolic space-time problem
with the same inf-sup and continuity constants for arbitrary properly chosen 
spaces $W_-$ and $W_+$. If $f(\cdot,\omega) \in 
L^2((0,T);W_+^*)$ and $u_0(\omega) \in W_0$, we obtain an analogous bound 
to \eqref{eq:normboundAS_nt2}:
\begin{equation}\label{eq:normboundAS_time_indep}
\begin{aligned}
\norm{L^2((0,T);W_-^*)}{u}^2 &\leq \Am^{-2} \norm{L^2((0,T);W_+^*)}{f}^2+
\Am^{-1}\norm{W_0}{u_0}^2, \quad \bP\mbox{\rm{-a.s.}},
\end{aligned}
\end{equation}
with $\Am$ as in \eqref{eq:bounds_for_A_shifted}.
We can conclude by Theorems~\ref{thm:bnbthm} and 
\ref{thm:reg_prop2} 
that the $\omega$-dependent part 
of the inf-sup constant behaves as $\Am^{-1}$,
and that the possible unboundedness of $\AM$ plays no 
role in the estimates for the norm of the solution, in contrast to
the general case expressed 
in Theorem~\ref{thm:bnbas}, due to the presence of $\rho$ there.

\section{Results of quasi-optimality}\label{sec:quasi_optimality}

\subsection{Petrov-Galerkin spatial 
semidiscretization}\label{subsec:semidiscrete}
This section is based on the results for the
Petrov-Galerkin spatial semidiscretization of the deterministic heat 
equation in \cite[Chapter~3]{Fra}, to which we refer for more details. We show 
that, for a spatial semidiscretization on 
finite dimensional subspaces $ V_h \subset V $ for which the $H$-orthogonal 
projection $P_h$ is bounded with respect to the $\norm{V}{\cdot}$-norm, 
the error satisfies a 
quasi-optimal inequality with a quasi-optimality constant which is proportional 
to $\rho(\omega)$. 

We start by introducing two different norms on the space 
$V_h^*$, namely:
\begin{equation*}
\begin{aligned}
&\norm{V^*}{v_h} = \sup_{\substack{w\in V\\\norm{V}{w} 
=1}}{\dual{v_h}{w}}, &&\norm{V_h^*}{v_h} &= \sup_{\substack{w_h\in 
V_h\\\norm{V}{w_h} 
=1}}{\duals{v_h}{w_h}{V_h^*}{V_h}}.
\end{aligned}
\end{equation*}
The two norms are equivalent, $ 
\norm{V_h^*}{v_h} \leq \norm{V^*}{v_h} \leq  c_h \norm{V_h^*}{v_h} $, 
but 
the constant $c_h:= \sup_{v_h \in 
V_h}{\frac{\norm{V^*}{v_h}}{\norm{V_h^*}{v_h}}}$ 
might in general not be uniform in the choice of the subspace $V_h$.
If we denote by $P_h$ the extension to $V^*$ of the $H$-orthogonal 
symmetric projection onto $V_h$, the following holds (see 
\cite[Proposition 3.2]{Fra}): 
\begin{equation}\label{eq:def_c_h}
c_h = \norm{\cL(V^*)}{P_h} = \norm{\cL(V^*)}{I-P_h} = 
\norm{\cL(V)}{I-P_h} = \norm{\cL(V)}{P_h}.
\end{equation}
We assume in the remainder of this article that the spatial mesh is such that 
$P_h$ is stable in the $V$-norm, that is, the validity of the following (see, 
e.g., \cite{Bank}):
\begin{align}\label{eq:c_h_unif_bdd}
c_h:= \sup_{v_h \in 
V_h}{\frac{\norm{V^*}{v_h}}{\norm{V_h^*}{v_h}}} < \infty, \quad \mbox{ for all 
} h. 
\end{align}
Within this new setting, we introduce the following semidiscrete 
spaces:
\begin{align*}
{\cY}^h:= L^2((0,T);V_h), \quad {\cX_{0,\{T\}}^h} := L^2((0,T);V_h) \cap 
H_{0,\{T\}}^1((0,T);V_h^*). 
\end{align*}
The $ \norm{{\cX}}{\cdot}$-norm on $\cX_{0,\{T\}}^h$ is modified as follows:
\begin{equation*}
\norm{{\cX}^h}{x_h}^2:= \norm{H}{x_h(0)}^2 + \int_0^T{(\norm{V}{x_h(t)}^2 +
\norm{V_h^*}{\dot{x}_h(t)}^2 )\,\dd t},
\end{equation*}
and similarly for the $\aabs{{\cX}_{\omega}^h}{\cdot}$-norm.
This replacement is necessary in order to adapt the proof of 
Lemma~\ref{thm:bnbas}
to the new framework. The following results hold.
\begin{theorem}\label{thm:bnbthm_sd}
Under the assumptions of Theorem~\ref{thm:bnbthm}, the bilinear form $\sB_{\omega}^*(\cdot,\cdot)$ satisfies the conditions 
\eqref{eq:BNB1A},
\eqref{eq:BNB2A}, and \eqref{eq:BDD} $\bP$-a.s., on the couple of subspaces $
(\cY^h, \aabs{\cY_{\omega}}{\cdot})$ and 
$(\cX_{0,\{T\}}^h,\aabs{\cX^{h}_{\omega}}{\cdot})$, with the same inf-sup 
constant and continuity constant as in 
Theorem~\ref{thm:bnbthm}.
\end{theorem}
Since the proof is essentially the same as in the continuous case, we refrain 
from presenting any more details here.
A direct consequence of the result above is the almost sure 
existence of a unique solution $u_h$ to the
semidiscrete problem:
\begin{equation} \label{eq:semispace}
u_h \in \cY^h \colon \sB_{\omega}^*(u_h,x_h) = \sF_{\omega}(x_h), 
\quad \forall
x_h \in {\cX^h_{0,\{T\}}},\quad {\rm a.e. } ~ \omega \in \Omega.
\end{equation}
The solution satisfies the same bound as its continuous counterpart in
\eqref{eq:normboundAS_nt2}:
\begin{equation*}
\begin{aligned}
\norm{L^2((0,T);V)}{u_h}^2 &\leq  
\Am^{-2} \norm{{L^2((0,T);V^*)}}{f}^2 +
  \Am^{-1} \norm{H}{u_0}^2, \quad \bP\mbox{\rm{-a.s.}}
\end{aligned}
\end{equation*}
In view of Theorem~\ref{thm:bnbthm_sd} we 
can now exploit the quasi-optimality theory and derive a bound for the error 
in terms of the quasi-optimality constant. Recall that $\rho = 
\frac{\AM}{\Am}$; the following lemma,
modification of \cite[Theorem 3.4]{Hackbusch}, holds.
 
\begin{theorem}\label{thm:quasi_optimality_sd}
If \eqref{eq:c_h_unif_bdd} holds, then the Galerkin method introduced in this 
section is quasi-optimal, and satisfies 
the estimate
\begin{equation}\label{eq:quasioptimalityEstimate_sd}
\norm{\cY}{u - u_h} \leq c_{h} (1 + \rho) 
\inf_{y_h \in \cY^h}{\norm{\cY^h}{u-y_h}}, \quad \bP\mbox{\rm{-a.s.}}
\end{equation}
\end{theorem}
\begin{proof}
We assume as usual that $\omega\in\Omega$ is given; for arbitrary 
$ y_h \in \cY^h $ we have 
\begin{align}
 u - u_h = u - R_{\cY^h}u  = (I-R_{\cY^h})u  = (I-R_{\cY^h})(u-y_h),
\end{align}
where $R_{\cY^h}u := u_h$ denotes the Ritz projection. 
By taking norms, we have
\begin{align}\label{eq:q_opt_sd2}
\norm{\cY}{ u - u_h} \leq \norm{\sL(\cY)}{I-R_{\cY^h}}  \norm{\cY}{u-y_h}.
\end{align}
In particular, since $R_{\cY^h}$ is an idempotent operator on a Hilbert 
space, it holds that $\norm{\sL(\cY)}{I-R_{\cY^h}}  = 
\norm{\sL(\cY)}{R_{\cY^h}}$ (see, e.g., \cite{XuZik}), so that in order to 
prove the claim all we need to to is to bound the norm of the Ritz projection.

By means of the
triangle inequality, we preliminary get
\begin{equation}\label{eq:trianglequasioptproof}
\norm{\cY}{ R_{\cY^h} u} = \norm{\cY}{u_h} \leq \norm{\cY}{u_h - P_h u} +
\norm{\cY}{P_h u}.
\end{equation}
We denote by $e_h$ the quantity $u_h - P_h u$. By using the 
definition of $P_h$, we have that for every $x_h\in \cX_{0,\{T\}}^h$
\begin{equation*}
\int_0^T{ \dual{-\dot{x}_h (t) }{e_h(t)}}\,\dd t  =  
\int_0^T{ \dual{-\dot{x}_h(t)}{u_h(t) - u(t)}}\,\dd t .
\end{equation*}
Since the right-hand side is the difference between the solutions of the 
continuous and the discrete problems, we obtain that
\begin{align*}
\int_0^T{ \dual{-\dot{x}_h(t)}{e_h(t)}}\,\dd t  
&=  \int_0^T{ \dual{-\dot{x}_h(t)}{u_h(t) - u(t)}}\,\dd t \\
&=  \int_0^T{ \dual{-\Ao x_h (t)}{u_h(t) - u(t)}}\,\dd t \\
&=  \int_0^T{\dual{\Ao(t)(u(t)-u_h(t))}{x_h(t)}}\,\dd t,
\end{align*}
which proves that $e_h \in H^1((0,T);V_h^*)$. The last expression can be 
rewritten as
\begin{equation}
\int_0^T{ \dual{-\dot{x}_h(t)}{e_h(t)}}\,\dd t  +
\int_0^T{ \dual{\Ao(t)(u_h(t)-u(t))}{x_h(t)}}\,\dd t = 0,
\end{equation}
which by adding and subtracting $P_h u$, becomes:
\begin{equation}\label{eq:beforeProofQOspace_sd}
\begin{aligned}
&\int_0^T{ \dual{-\dot{x}_h(t)}{e_h(t)}}\,\dd t  +
\int_0^T{ \dual{\Ao(t)(u_h(t) - P_h u(t))}{x_h(t)}}\,\dd t 
\\&\qquad = \int_0^T{ 
\dual{\Ao(t)(u(t) - P_h u(t) )}{x_h(t)}}\,\dd t.
\end{aligned}
\end{equation}
If we integrate by parts in \eqref{eq:beforeProofQOspace_sd}, for 
every $x_h\in \sC^{\infty}_0((0,T);V_h)$ we have:
\begin{equation}\label{eq:equalityweakderivativeerror}
\begin{aligned}
&\int_0^T{\Big(  \dual{\dot{e}_h(t)}{x_h(t)} + \dual{\Ao(t)e_h(t)}{x_h(t)}  
\Big)}\,\dd 
t\\
&\qquad \qquad \qquad = \int_0^T{ 
\dual{\Ao(t)(u(t)-P_h{u(t)})}{x_h(t)}}\,\dd 
t.
\end{aligned}
\end{equation}
By density, the formula in \eqref{eq:equalityweakderivativeerror} holds also 
for 
every $x_h\in L^{2}((0,T);V_h)$. 

By testing with $x_h = (T-t)\phi_h$, $\phi_h \in V_h$, integrating by parts 
directly 
in \eqref{eq:beforeProofQOspace_sd} and 
subtracting the
resulting expression from \eqref{eq:equalityweakderivativeerror}, we deduce 
that
$e_h(0)=0$. Testing now \eqref{eq:equalityweakderivativeerror} with 
$x_h=e_h$, we obtain that
\begin{equation*}
\frac{1}{2}\norm{H}{e_h(T)}^2 + 
\int_0^T{ \dual{\Ao (t) e_h(t)}{e_h(t)}  }\,\dd t
\leq
\AM \norm{\cY}{u-P_hu}\norm{\cY}{e_h}.
\end{equation*}
This, in turns, implies that
\begin{equation*}
\begin{aligned}
\Am \norm{\cY}{e_h}^2 &\leq
\AM \norm{\cY}{u-P_hu}\norm{\cY}{e_h} ,
\end{aligned}
\end{equation*}
that is,
\begin{equation}\label{eq:finalestimatequasiopt}
\norm{\cY}{u_h - P_hu}  \leq \frac{\AM}{ \Am }
\norm{\cY}{u-P_hu}.
\end{equation}
Using \eqref{eq:finalestimatequasiopt} in \eqref{eq:trianglequasioptproof}, and 
then \eqref{eq:def_c_h}, we obtain that 
\begin{align*}
\norm{\cY}{u_h} &\leq  c_h \Big(1 + \frac{\AM}{ \Am }\Big)\norm{\cY}{u},
\end{align*}
which proves that the norm of the Ritz projection is bounded as follows:
\begin{align}\label{eq:q_opt_sd1}
\norm{\sL(\cY)}{R_{\cY^h}} \leq c_h \Big(1 + 
\frac{\AM}{ \Am }\Big).
\end{align}
Since $\omega \in \Omega$ arbitrary, the claim follows by combining 
\eqref{eq:q_opt_sd1} and \eqref{eq:q_opt_sd2}.
\end{proof}

\subsection{Petrov-Galerkin full discretization}\label{subsec:discrete}
This section is based on the results for the
Petrov-Galerkin discretization of the deterministic heat equation presented in  
\cite{Fra}, \cite{UrbanPatera2}, \cite{Andreev1}, and \cite{LarssonMolteni2}, to which we 
refer for more details. We show that, for a full discretization based 
on finite-dimensional tensor spaces, the error satisfies a quasi-optimal 
inequality, with a quasi-optimality constant which not only depends on $\rho$, 
but also on a term proportional to $\AM$.

We consider a partition
of the time interval $[0,T]$, given by $\cT_k = \{ t_{i-1} < t \leq t_{i}, 
\,i=1,\ldots,N\}$, where $k_i := {t_{i+1}-t_{i}}$, and $k = 
{\max}_i {k_i}$. We denote by $S_{k,q+1}$ 
the 
space of continuous functions of $t$ that are piecewise polynomials 
of degree at most $q+1$ and by $Q_{k,q}$ the space 
of functions which are piecewise 
polynomials of degree at most $q$. We define 
the finite-dimensional subspaces 
$\ctiiY := Q_{k,q} \otimes 
V_h $, and $ \ctiiX_{0,\{T\}} := \{ X \in S_{k,q+1}
\otimes V_h \colon X(T) = 0\}$, for some finite-dimensional subspace $V_h 
\subset V$. We assume in the remainder of this section that $A$ is as in 
\S~\ref{subsec:inf_sup_alt} and  does not depend on time.
The discretized problem can be written in variational form as
\begin{equation} \label{eq:discrete_form_q}
U \in {\ctiiY}\colon \sB_{\omega}^*(U,X) = 
\sF_{\omega}(X), 
\quad \forall X \in
{\ctiiX_{0,\{T\}}} , \quad {\rm a.e. } ~ \omega \in \Omega.
\end{equation}
We introduce the following norm:
\begin{align}\label{eq:mod_norm}
\aabs{\ctiiXo}{X}^2&:= 
\sum_{i=0}^{N-1}{\int_{I_i}{\Big(\norm{H}{\Ao^{-\frac12}\dot{X}(t)}^2 
+
\norm{H}{\Ao^{\frac{1}{2}}\Pi^{(q)}_i X(t)}^2 \Big)\,\dd t}} + \norm{H}{X(0)}^2,
\end{align}
where $\Pi_i^{(q)}$ is locally defined on each $I_i$ as the orthogonal 
$L^2$-projection onto the space of polynomials of degree at most $q$.
In \cite{AndreevPhd}, it is shown that the norm
\begin{align}
\aabs{\cX_k}{X}^2&:= 
\sum_{i=0}^{N-1}{\int_{I_i}{\Big(\norm{V^*}{\dot{X}(t)}^2 
+
\norm{V}{\Pi^{(q)}_i X(t)}^2 \Big)\,\dd t}} + \norm{H}{X(0)}^2.
\end{align}
is equivalent to the norm $\norm{\cX}{\cdot}$, but the equivalence is 
not uniform in the choice of $h$ and $k$ unless the following CFL condition is 
assumed to hold:
\begin{align}\label{eq:CFL}
c_S = k \sup_{v \in V_h}{\frac{ \norm{V}{v}\hfill }{ 
\norm{V^*}{v}}}  < \infty, \qquad \mbox{for all $h$ and $k$}.
\end{align}
Similar conclusions hold for the norm 
$\aabs{\ctiiXo}{\cdot}$ and the norm $\aabs{\cX_{\omega}}{\cdot}$, but the 
equivalence 
constant is now also $\omega$-dependent, as shown in the next lemma, 
based on \cite[Proof of (5.2.63)]{AndreevPhd}.
\begin{lemma}\label{lemma:cfl_omega}
The following relationship holds for any $X \in 
\ctiiX_{0,\{T\}}$,  $\bP$-a.s.:
\begin{align}\label{eq:equiv_norm_cX}
\aabs{\ctiiXo}{X} \leq \aabs{ \cX_{\omega}}{X}  \leq C (1 + \AM c_{S} )
\aabs{\ctiiXo}{X}.
\end{align}
More precisely,
\begin{equation}\label{eq:equiv_norm_cX_expl}
\begin{aligned}
&\norm{H}{X(0)}^2 + \int_{I_i} \norm{H}{\Ao^{-\frac12} \dot{X}(t) }^2\,\dd t + 
\int_{I_i} \norm{H}{\Ao^{\frac12} {X}(t) }^2\,\dd t \\
&\quad \leq \norm{H}{X(0)}^2 
+(1+ c_{S,\omega}^2) 
\int_{I_i} \norm{H}{\Ao^{-\frac12} \dot{X}(t) }^2\,\dd t + 
\int_{I_i} 
\norm{H}{\Ao^{\frac12} \Pi^{(q)}{X}(t) }^2\,\dd t,
\end{aligned}
\end{equation}
where $c_{S,\omega}^2 \leq \frac{\AM^2}{12} c_S^2$ .
\end{lemma}
\begin{proof}
On each subinterval $I_i$ we can represent $X\in \ctiiX_{0,\{T\}} $ as follows:
\begin{align*}
X\big|_{I_i} = \sum_{j=0}^{q+1} L_j^i(t) \otimes v_{j,h},
\end{align*}
for some $ v_{j,h} \in V_h$ and where $L_j^i(t)$ is the Legendre polynomial 
of degree $j$ on the interval $I_i$.
The following estimate holds:
\begin{align*}
\int_{I_i} \norm{H}{\Ao^{-\frac12} \dot{X}(t) }^2\,\dd t \geq \frac{6}{k}  
\norm{H}{\Ao^{-\frac12} v_{q+1,h}}^2.
\end{align*}
The mutual $L^2$-orthogonality of Legendre polynomials leads to
\begin{align*}
\int_{I_i} \norm{H}{\Ao^{\frac12} {X}(t) }^2\,\dd t = 
\frac{k}{2}\norm{H}{\Ao^{\frac12} v_{q+1,h}}^2 + \int_{I_i} 
\norm{H}{\Ao^{\frac12} \Pi^{(q)}{X}(t) }^2\,\dd t,
\end{align*}
which is
\begin{align*}
\int_{I_i} \norm{H}{\Ao^{\frac12} {X}(t) }^2\,\dd t &= 
\frac{k}{2} \frac{ \norm{H}{\Ao^{\frac12} v_{q+1,h}}^2}{\norm{H}{\Ao^{-\frac12} 
v_{q+1,h}}^2} \norm{H}{\Ao^{-\frac12} v_{q+1,h}}^2  + \int_{I_i} 
\norm{H}{\Ao^{\frac12} \Pi^{(q)}{X}(t) }^2\,\dd t\\
&\leq
\frac{k^2}{12} \frac{ \norm{H}{\Ao^{\frac12} 
v_{q+1,h}}^2}{\norm{H}{\Ao^{-\frac12} 
v_{q+1,h}}^2} \int_{I_i} \norm{H}{\Ao^{-\frac12} \dot{X}(t) }^2\,\dd t + 
\int_{I_i} 
\norm{H}{\Ao^{\frac12} \Pi^{(q)}{X}(t) }^2\,\dd t.
\end{align*}
If we introduce the notation
\begin{align*}
c^2_{S,\omega}:= \frac{k^2}{12} \frac{ \norm{H}{\Ao^{\frac12} 
v_{q+1,h}}^2}{\norm{H}{\Ao^{-\frac12} 
v_{q+1,h}}^2},
\end{align*}
we can see that 
\begin{align}\label{eq:last_inequality}
\int_{I_i} \norm{H}{\Ao^{\frac12} {X}(t) }^2\,\dd t \leq c_{S,\omega}^2 
\int_{I_i} \norm{H}{\Ao^{-\frac12} \dot{X}(t) }^2\,\dd t + 
\int_{I_i} 
\norm{H}{\Ao^{\frac12} \Pi^{(q)}{X}(t) }^2\,\dd t,
\end{align}
where the following holds:
\begin{align}\label{eq:bounds_CFL_omega}
c_{S,\omega}^2 \leq \frac{k^2}{12} \frac{\AM 
\norm{V}{v_{q+1,h}}^2}{\AM^{-1}\norm{V^*}{v_{q+1,h}}^2} \leq 
\frac{\AM^2}{12}c_S^2.
\end{align}
By adding $\int_{I_i} \norm{H}{\Ao^{-\frac12} \dot{X}(t) }^2\,\dd t$ to both 
sides of \eqref{eq:last_inequality}, summing over each subinterval and adding 
$\norm{H}{X(0)}^2$, we obtain \eqref{eq:equiv_norm_cX_expl}. The first part of 
the claim follows by using \eqref{eq:bounds_CFL_omega} in 
\eqref{eq:equiv_norm_cX_expl}.
\end{proof}
We assume in the rest of this section that \eqref{eq:CFL} holds true.
The following discrete counterpart to Theorem~\ref{thm:bnbthm} holds, for 
each $\omega$:

\begin{theorem}\label{thm:bnbthm_fd}
The bilinear form appearing in \eqref{eq:discrete_form_q} satisfies the 
following:
\begin{align}
&C_{B}^{h}:=\sup_{0\neq Y \in \ctiiY }\sup_{0\neq X \in \ctiiX_{0,\{T\}}} 
\frac{\sB_{\omega}^*(Y,X)}{ \aabs{\cY_{\omega}}{Y} \aabs{\ctiiXo}{X}} = 
1,\quad 
\bP\mbox{\rm{-a.s.}}, \label{eq:bdd_fd}\\
&c_{B}^{h}:=\inf_{0\neq Y \in \ctiiY }\sup_{0\neq X \in \ctiiX_{0,\{T\}}} 
\frac{\sB_{\omega}^*(Y,X)}{
\aabs{\cY_{\omega}}{Y} \aabs{\ctiiXo}{X}  } = 1,\quad 
\bP\mbox{\rm{-a.s.}},\label{eq:infsup_fd}\\
&\forall X\in \ctiiX_{0,\{T\}}  \qquad \sup_{0\neq Y \in 
\ctiiY} \sB_{\omega}^*(Y,X)>0, \quad 
\bP\mbox{\rm{-a.s.}}\label{eq:infsup2_fd}
\end{align}
\end{theorem}
The proof of this is omitted, since once $\omega$ is fixed, we can refer 
for example to \cite{UrbanPatera} or \cite{UrbanPatera2}, modified in the 
spirit of \S~\ref{subsec:inf_sup_alt}.
Similarly to the continuous case, and by 
using Lemma~\ref{lemma:cfl_omega}, we can 
thus prove that there exists a unique solution to problem 
\eqref{eq:discrete_form_q}, and that its norm satisfies the following:
\begin{equation}\label{eq:norm_sol_discrete}
\begin{aligned}
\Am \norm{ L^2((0,T);V) }{U}^2   &\leq
 C\big(1 + \AM^2 c^2_S\big)\Am^{-1} \norm{L^2((0,T);V^*)}{f}^2 
+ \norm{H}{u_0}^2,\quad 
\bP\mbox{\rm{-a.s.}},
\end{aligned}
\end{equation}
which is
\begin{equation}\label{eq:normboundAS_nt2_dis}
\begin{aligned}
\norm{L^2((0,T);V)}{U}^2 &\leq C\big( \Am^{-2} + \rho^2 
c^2_S\big) \norm{L^2((0,T);V^*)}{f}^2 
+  \Am^{-1} \norm{H}{u_0}^2, \quad \bP\mbox{\rm{-a.s.}}
\end{aligned}
\end{equation}

The results of existence and uniqueness, and the bounds obtained for the 
bilinear form in Theorems~\ref{thm:bnbthm} and \ref{thm:bnbthm_fd} can be used 
to derive results of quasi-optimality for the error $u-U$. We start by showing 
a sharp connection between the quasi-optimality constant and the random 
variable $c_{S,\omega}$ appearing in Lemma \ref{lemma:cfl_omega}.
\begin{theorem}
The quasi-optimality constant $q_S$ with respect to $\|\cdot\|_{\cY_{\omega}}$ 
for the 
Petrov-Galerkin discretization introduced 
in this section is a random variable, defined for every $\omega$ as the 
smallest $q_S$ for which the 
following inequality holds:
\begin{align}\label{eq:quasi_optimality}
\norm{\cY_{\omega}}{u-U}\leq q_S(\omega) \inf_{ Y \in \ctiiY 
}\norm{\cY_{\omega}}{u-Y}.
\end{align}
For every $\omega$, it satisfies the following 
bound:
\begin{align}
q_S \leq \sqrt{1+c_{S,\omega}^2} 
\leq C \max\{\AM,1\}\sqrt{1+c_{S}^2}.
\end{align}
\end{theorem}
\begin{proof}
The quasi-optimality constant $q_S$ is bounded as
follows (see for example \cite[Corollary 1.3]{Fra}):
\begin{equation}\label{doubleboundq_S}
q_S \leq
\frac{C_{\cY \times \cX^{h,k,q+1}_{0,\{T\}}}}{c_{B}^{h}}.
\end{equation}
Here $C_{\cY \times \cX^{h,k,q+1}_{0,\{T\}}}$ denotes the continuity constant of
 $\sB_{\omega}^*(\cdot,\cdot)$ on $ \cY \times \cX^{h,k,q+1}_{0,\{T\}}$, with 
the 
latter space endowed with the $\aabs{\ctiiXo}{\cdot}$-norm.
Since we already know that $c_{B}^{h} = 1$, in order to prove the claim it 
suffices to bound 
$C_{\cY 
\times \cX^{h,k,q+1}_{0,\{T\}}}$:
\begin{align*}
C_{\cY \times \cX^{h,k,q+1}_{0,\{T\}}} &= \sup_{y\in\cY}\sup_{X \in 
\cX^{h,k,q+1}_{0,\{T\}}
} \frac{\sB_{\omega}^*(y,X)}{ \aabs{\cY_{\omega}}{y} \aabs{\ctiiXo}{X}}\\
&\leq \Big( \sup_{X \in
\cX^{h,k,q+1}_{0,\{T\}}}\frac{\aabs{\cX_{\omega}}{X}\hfill}{\aabs{\ctiiXo}{X}}  
\Big) \Big( \sup_{y\in\cY}\sup_{X \in
\cX^{h,k,q+1}_{0,\{T\}}}\frac{\sB_{\omega}^*(y,X)}{
\aabs{\cY_{\omega}}{y} \aabs{\cX_{\omega}}{X}  }  
\Big)\\
&\leq C_{B} \sqrt{1+c_{S,\omega}^2}.
\end{align*}
By inserting this last bound in Equation~\eqref{doubleboundq_S}, and 
using Lemma~\ref{lemma:cfl_omega}, we obtain:
\begin{equation*}
q_S \leq \frac{ C_{B} }{c_{B}^{h}} \sqrt{1+c_{S,\omega}^2} \leq C 
\max\{\AM,1\}\sqrt{1+c_{S}^2},
\end{equation*}
which proves the claim.
\end{proof}
This leads to the following result of quasi-optimality:
\begin{theorem}\label{thm:quasi_optimality_fd}
The Petrov-Galerkin method introduced above is quasi-optimal and satisfies the 
estimate
\begin{equation}\label{quasioptimalityEstimate_fd}
\aabs{\cY_{\omega}}{u - U} \leq \sqrt{1+c_{S,\omega}^2}
\inf_{Y\in\ctiiY}{\aabs{\cY_{\omega}}{u-Y}}, \quad 
\bP\mbox{\rm{-a.s.}},
\end{equation}
and hence:
\begin{equation}\label{quasioptimalityEstimate_fd1}
\norm{\cY}{u - U} \leq C \rho \sqrt{1+\AM^2 c_S^2} 
\inf_{Y\in\ctiiY}{\norm{\cY}{u-Y}},\quad 
\bP\mbox{\rm{-a.s.}}
\end{equation}
\end{theorem}

\section{$L^p(\Omega)$-estimates} \label{sec:Lpest}
We start this section by presenting some sufficient conditions to have 
a solution $u$ to \eqref{eq:secondspacetime} (respectively a solution $u_h$ 
to \eqref{eq:semispace} and a solution $U$ to \eqref{eq:discrete_form_q}) 
which belongs to in $L^p(\Omega;\cY)$, $p\geq1$. 

\subsection{$L^p(\Omega)$-estimates for the solution}\label{subsec:Lpest}
Provided the almost sure existence of a solution to \eqref{eq:secondspacetime}, 
we want to give some sufficient condition on the existence of the moments of the
data and of the two random variables $\AM$ and $\Am$, bounding the operator $A$,
such that the $p$-moments of the solutions exist, for some $p\in[1,\infty]$.


\begin{theorem}\label{thm:Lptheorem2}
Under the assumptions of Theorem \ref{thm:bnbthm} and if
\begin{enumerate}
 \item $ f \in L^{\alpha}(\Omega;L^2((0,T);V^*))$, 
 \item $u_0 \in L^{\beta}(\Omega;H)$,
 \item ${\Am^{-1}} \in L^{\gamma}(\Omega;\bR)$,
\end{enumerate}
for some exponents $\alpha$, $\gamma$, $\beta \geq1$, such that  $p:=  \min \{ 
\frac{\alpha\gamma}{\alpha + \gamma} , \frac{2\beta\gamma}{\beta + 2\gamma}\} 
\geq 1$, then both $u$, solution to \eqref{eq:secondspacetime}, and $u_h$, 
solution to \eqref{eq:semispace}, belong to
$L^{p}(\Omega;\cY)$. The same holds for $U$, solution to 
\eqref{eq:discrete_form_q}, if we assume that there is a constant $\RM$ 
such that
\begin{align}\label{eq:assumption_ratio}
\RM := \sup_{\omega \in \Omega}\abs{\rho(\omega)} < \infty.
\end{align}
\end{theorem}
\begin{proof}
The solutions $u_h$, $U$, and $u$ all satisfy the same bounds, up to a factor 
$\rho(\omega)c_S$ in the case of $U$. Since we assume 
\eqref{eq:assumption_ratio}, we have
$\rho(\omega)c_S \leq \RM c_S$, which does not depend on $\omega$ and we 
therefore only need to prove this theorem for $u$. By means of H\"{o}lder's 
inequality and \eqref{eq:normboundAS_nt2}, 
we have that:
\begin{equation}\label{eq:estimate1_2}
\begin{aligned}
\norm{L^p(\Omega;L^2((0,T);V))}{u} &\leq
\norm{L^p(\Omega;L^2((0,T);V^*))}{ \Am^{-1} f} 
+
\norm{L^p(\Omega;H)}{ \Am^{-\frac12} u_0}
.
\end{aligned}
\end{equation}
If we focus on the first term of the right-hand side, we have:
\begin{align}
\norm{L^p(\Omega;L^2((0,T);V))}{ \Am^{-1} f} 
\leq 
\norm{L^{pq}(\Omega;\bR)}{ \Am^{-1} } 
\norm{L^{pq'}(\Omega;L^2((0,T);V^*))}{f},
\end{align}
for any pair of dual H\"{o}lder exponents $q$ and $q'$. If we choose $q'= 1 + 
\frac{\alpha}{\gamma}$ and
$q =  1 + \frac{\gamma}{\alpha} $, we notice that for $p:= 
\frac{\alpha\gamma}{\alpha + \gamma}$ we have $pq'=\alpha $ and $pq 
=\gamma$, so that \eqref{eq:estimate1_2} gives:
\begin{align}\label{eq:Lp_sol_first_step}
\norm{L^p(\Omega;L^2((0,T);V))}{u}  \leq
\norm{L^{\gamma}(\Omega;\bR)}{ \Am^{-1} } 
\norm{L^{\alpha}(\Omega;L^2((0,T);V^*))}{f} < \infty.
\end{align}
Similarly, the second term of the right hand side gives:
\begin{align*}
\norm{L^p(\Omega;H)}{ \Am^{-\frac12}  u_0} 
&\leq 
\norm{L^{pq}(\Omega;\bR)}{\Am^{-\frac12}}
\norm{L^{pq'}(\Omega;H)}{u_0} =
\norm{L^{\frac{pq}{2}}(\Omega;\bR)}{ \Am^{-1} }^{\frac12} 
\norm{L^{pq'}(\Omega;H)}{u_0}.
\end{align*}
If we choose $q'= 1 + \frac{\beta}{2\gamma}$ and
$q =  1 + \frac{2\gamma}{\beta} $, we notice that for  $p:= 
\frac{2\beta\gamma}{\beta + 2\gamma}$  we have 
$pq'=\beta $ and $\frac{pq}{2} =\gamma$
so that \eqref{eq:estimate1_2} gives:
\begin{align}\label{eq:Lp_sol_second_step}
\norm{L^p(\Omega;H)}{\Am^{-\frac12} 
u_0} \leq 
\norm{L^{\gamma}(\Omega;\bR)}{ \Am^{-1} }^{\frac12} 
\norm{L^{\beta}(\Omega;H)}{u_0} < \infty.
\end{align}
Since $\Omega$ is a probability space, it suffices to take $ p = \min \{ 
\frac{\alpha\gamma}{\alpha + \gamma} , \frac{2\beta\gamma}{\beta + 2\gamma}\}$ 
to have finite quantities both in \eqref{eq:Lp_sol_first_step} and 
\eqref{eq:Lp_sol_second_step}, so that the claim is proved.
\end{proof}

\begin{remark}\label{notesaboutunifbound}
Notice that, if we assume that $\Am^{-1}$ is uniformly bounded,
i.e., that $\gamma=\infty$, we can see from the first step of the proof
that the finiteness of the $p$-moments of the solution $u$ coincides to the one
of the $p$-moments of the initial data $u_0$ and of the load function $f$.
\end{remark}

Despite the fact that the estimates might not be sharp in general, our results 
are of 
relevance because they
allow the analysis of very general elliptic operators of the form
\begin{align}\label{eq:example_operator_omega_time_space}
A(t,\omega)u :=-\partial_i\Big( X_{ij}(t,\omega,\xi) \partial_j u \Big) + 
Y_{j}(t,\omega,\xi)\partial_j u + Z_{j}(t,\omega,\xi) u,
\end{align}
where the coefficients $X_{i,j}$, $Y_j$ and $Z$ are possibly unbounded 
time-dependent 
random fields and where $\xi$ denotes the spatial variable.
However, if we know explicitly how the quantities 
involved depend on $\omega$, a better analysis could be performed case by case, 
using \eqref{eq:normboundAS_nt2} as a starting point. The following 
example helps to 
clarify the possible limitations of the previous theorem.
\begin{example}\label{example:factorization_op}
Suppose that $\Lambda$ is a bounded 
domain 
in $ \bR^n$, $ V = H^1_0(\Lambda)$, $H =L^2(\Lambda)$, $\Omega =[0,1]$, 
$\zeta_0 \neq \zeta_1 \in 
[0,1]$, $\alpha \in (0,1)$, and 
assume that $\bP$ is the Lebesgue measure and that the operator $A$ is of the 
form
\begin{align*}
A(t,\omega) = A(\omega) := -\abs{\omega-\zeta_0}^{\alpha} 
\Delta.
\end{align*}
We take $u_0=0$ and assume that $f$, for some $g\in L^2((0,T); 
H^{-1}(\Lambda))$, is given by
\begin{align*}
f(t,\omega,\xi) :=  \abs{\omega-\zeta_1}^{-\alpha} g(t,\xi), 
\end{align*}
We have that $f \in 
L^p(\Omega;L^2((0,T);V^*))$ and $ \Am^{-1} \in 
L^p(\Omega;\bR)$ for any $ p <\frac{1}{\alpha}$.
For each fixed $\omega \in [0,1]$ the following point-wise estimate holds:
\begin{align*}
\norm{L^2((0,T);H_0^{1}(\Lambda))}{u(\cdot,\omega)} \lesssim  
\abs{\omega-\zeta_0}^{-\alpha} \abs{\omega-\zeta_1}^{-\alpha} 
\norm{L^2((0,T);H^{-1}(\Lambda))}{g}.
\end{align*}
Since the two singularities occur at different points, it is easy to see 
that the left-hand side 
belongs to $L^p(\Omega;\cY)$, for any $ p <\frac{1}{\alpha}$ as well. This is 
a sharper result than what 
Theorem~\ref{thm:Lptheorem2} would ensure, that is, 
$ u \in L^p(\Omega;\cY)$ 
for $ p < \frac{1}{2 \alpha} $.
\end{example}

There are of course many cases of interests where a local analysis as the one 
above is not available, for example when a factorization in a random part and a 
deterministic part is not available, or when $\Omega$ does not have compact 
support and both $\Am^{-1}$ and $f$ are heavy-tailed distributed 
(e.g., for ``power law'' distributions such as Pareto, which are of particular 
importance in finance, or Student's t-distribution).


\subsection{$L^p(\Omega)$-estimates for the semidiscrete 
error}\label{subsec:Lpest_err_sd}
We assume that the data characterizing 
the original problem \eqref{eq:randHeatStrong} is such 
that Theorem~\ref{thm:Lptheorem2} is applicable for a given $p>1$. Moreover, we 
denote 
by $u_h$ the spatial semidiscrete solution to the original problem, as 
in \S~\ref{subsec:semidiscrete}.

\begin{theorem}\label{thm:quasi_optimality_Lp_semi}
If there exist exponents $\alpha,\beta,\gamma,p$ as in the assumptions of 
Theorem~\ref{thm:Lptheorem2}, and if \eqref{eq:c_h_unif_bdd} 
and 
\eqref{eq:assumption_ratio} 
both hold, it follows that
\begin{align*}
 \norm{L^{p}(\Omega;\cY)}{u - u_h} &\leq c_h (1+\RM) \bE\Big[\inf_{y_h\in
L^p(\Omega;\cY^h)}{\norm{\cY}{u-y_h}^p}
 \Big]^{\frac{1}{p}}.
\end{align*}
\end{theorem}
\begin{proof}
Theorem~\ref{thm:quasi_optimality_sd} ensures the validity of 
\eqref{eq:quasioptimalityEstimate_sd} for every given $\omega$; in particular 
we 
can choose a different $y_h$ for every $\omega$, so that $y_h = y_h(\omega)$. 
If we 
now take the $L^p(\Omega;\cdot)$-norm at both sides of 
\eqref{eq:quasioptimalityEstimate_sd}, assume that $y_h \in L^p(\Omega;\cY^h)$, 
use \eqref{eq:assumption_ratio} to estimate $\rho$ in terms of an absolute 
constant, and use Theorem \ref{thm:Lptheorem2} to ensure the finiteness of both 
right- and left-hand sides, the claim follows. 
\end{proof}

Despite the fact that Example~\ref{example:factorization_op} suggests that 
our results are not the sharpest one could get, it is worth noticing that under 
relatively general assumption 
we can treat interesting cases not covered so far by other articles on the same 
topic (see, e.g., \cite{SchwabAndreev}).
\begin{example}
If we assume that $A$ is uniformly coercive with 
respect to $\omega$, but not uniformly bounded, and that 
\eqref{eq:assumption_ratio} holds (as for example if $\Am(\omega) \asymp 
\AM(\omega) \asymp 1 + \frac{1}{\abs{\omega}} $, with $\bP$ being the Lebesgue 
measure and $\Omega = [-0.5,0.5]$), an 
application of Theorem \ref{thm:Lptheorem2} ensures that a 
mean-square integrable solution exists, given that the data are mean-square 
integrable as 
well. Furthermore, Theorem \ref{thm:quasi_optimality_Lp_semi} ensures results 
also of 
quasi-optimality in a mean-square sense. The unboundedness of the operator $A$ 
with respect to the parameter/random variable $\omega$ does not affect the 
existence of a semidiscrete solution, nor the quasi-optimality estimates (and 
hence the rate of convergence of the error).
\end{example}


\subsection{$L^p(\Omega)$-estimates for the fully
discrete error}\label{subsec:Lpest_err_fd}
As in the previous subsection, we still assume that the 
data characterizing 
the original problem \eqref{eq:randHeatStrong} are such 
that Theorem~\ref{thm:Lptheorem2} is applicable for a given $p>1$.


\begin{theorem}\label{thm:quasi_optimality_Lp_disc}
If there exist exponents $\alpha,\beta,\gamma,p$ as in the assumptions of 
Theorem~\ref{thm:Lptheorem2}, and if \eqref{eq:CFL} 
and 
\eqref{eq:assumption_ratio} 
both hold, it follows that
\begin{align*}
 \norm{L^{\bar{p}}(\Omega;\cY)}{u - U} &\lesssim c_S \RM 
\bE[\AM^{\theta}]^{\frac{1}{\theta}} \bE\Big[\inf_{Y \in
L^p(\Omega;\ctiiY )}{\norm{\cY}{u-Y}^p}
 \Big]^{\frac{1}{p}},
\end{align*}
where $\AM \in L^{\theta}(\Omega;\bR)$, $\bar{p} = p - \frac{p^2}{\theta + p}$, 
and $\theta,\bar{p} \geq1$.
\end{theorem}
\begin{proof}
Theorem~\ref{thm:quasi_optimality_fd} ensures the validity of 
\eqref{quasioptimalityEstimate_fd1} for every given $\omega$. We proceed as in 
Theorem~\ref{thm:quasi_optimality_Lp_semi}, and estimate $\rho$ in terms of an 
absolute constant. H\"older's inequality with exponents 
$r=\frac{\theta+ p}{p}$, $r'=\frac{\theta+ p}{\theta}$, together with the 
assumption on $\AM$ and Theorem \ref{thm:Lptheorem2} ensure that the right-hand 
side is finite, thus proving the claim. 
\end{proof}
It is clear that, although in general $\bar{p} < p$, in the case 
of $\theta = \infty$ we have $\bar{p}=p$, a 
quasi-optimality constant uniformly bounded with respect to $\omega$, and 
quasi-optimal estimates in the same $L^p$ space where the 
discrete and continuous solutions live.

Results in the spirit of this section without a CLF condition for certain 
discretizations with 
piecewise linear and continuous functions can be found in \cite{Mo_Diss}.

\section{Numerical experiments}\label{sec:num_res}
We present some numerical experiments to give a preliminary 
indication of the sharpness of our estimates. We investigate the problem:
\begin{equation} \label{eq:randHeatStrong_numerics}
\begin{cases}
&\dot{u}(t,\omega) - a(\omega)\Delta u(t,\omega) = c_0(\omega)g(t),\\
&u(0,\omega) = 0,
\end{cases}
\end{equation}
where the coefficients $a(\omega)$, $c_0(\omega)$, and the function $g(t)$ 
are specified from case to 
case, the temporal domain is $[0,1]$, the spatial domain is $\Lambda 
= [0,1]^2$, and the boundary conditions are of homogeneous Dirichlet type.

\begin{figure}[h!]
        \centering
        \begin{subfigure}[b]{0.49\textwidth}
	  \includegraphics[width=\textwidth]{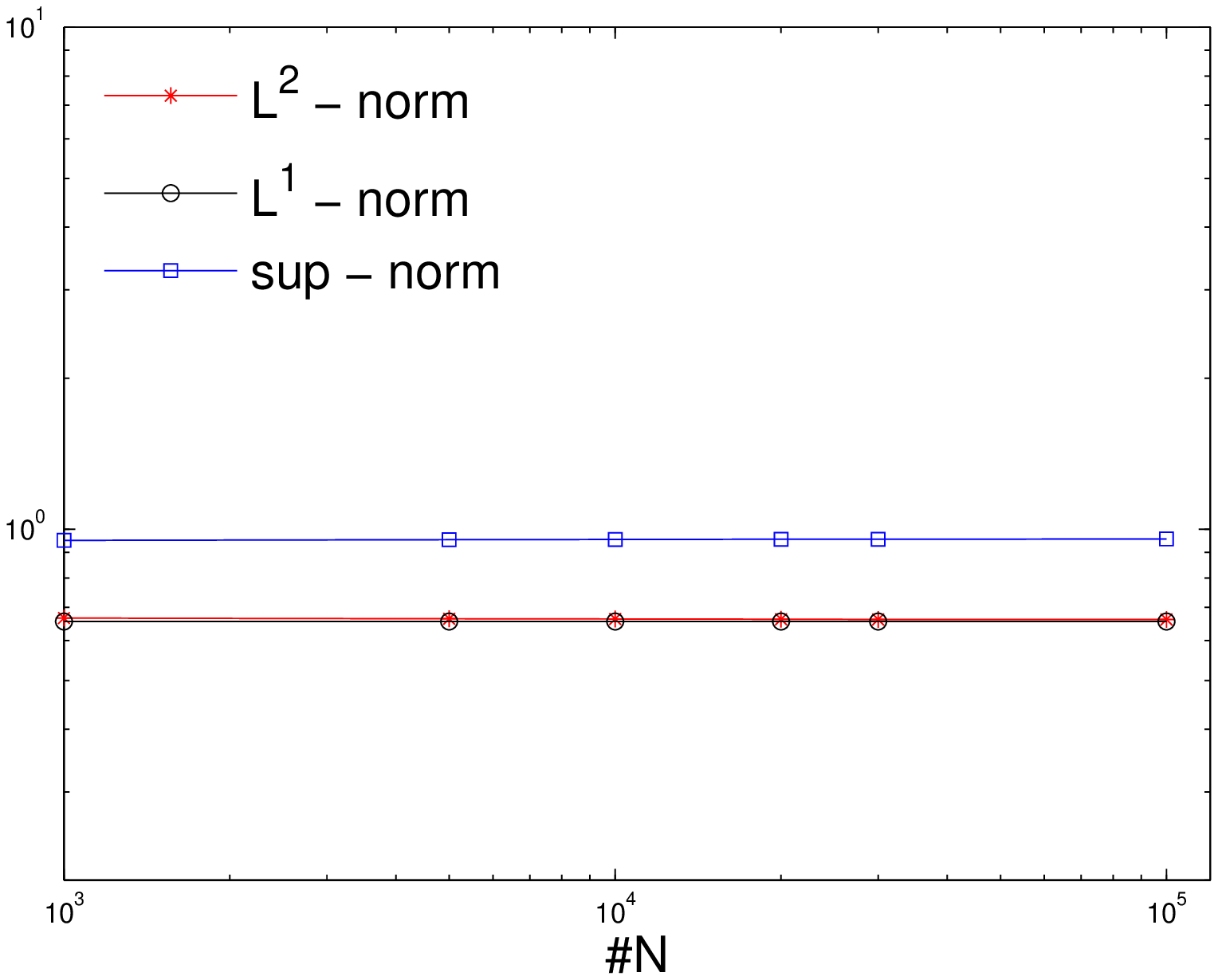}
	  \caption{Case \S~\ref{Exmpl1}}
	  \label{fig:case2}
        \end{subfigure}
        \begin{subfigure}[b]{0.49\textwidth}
	  \includegraphics[width=\textwidth]{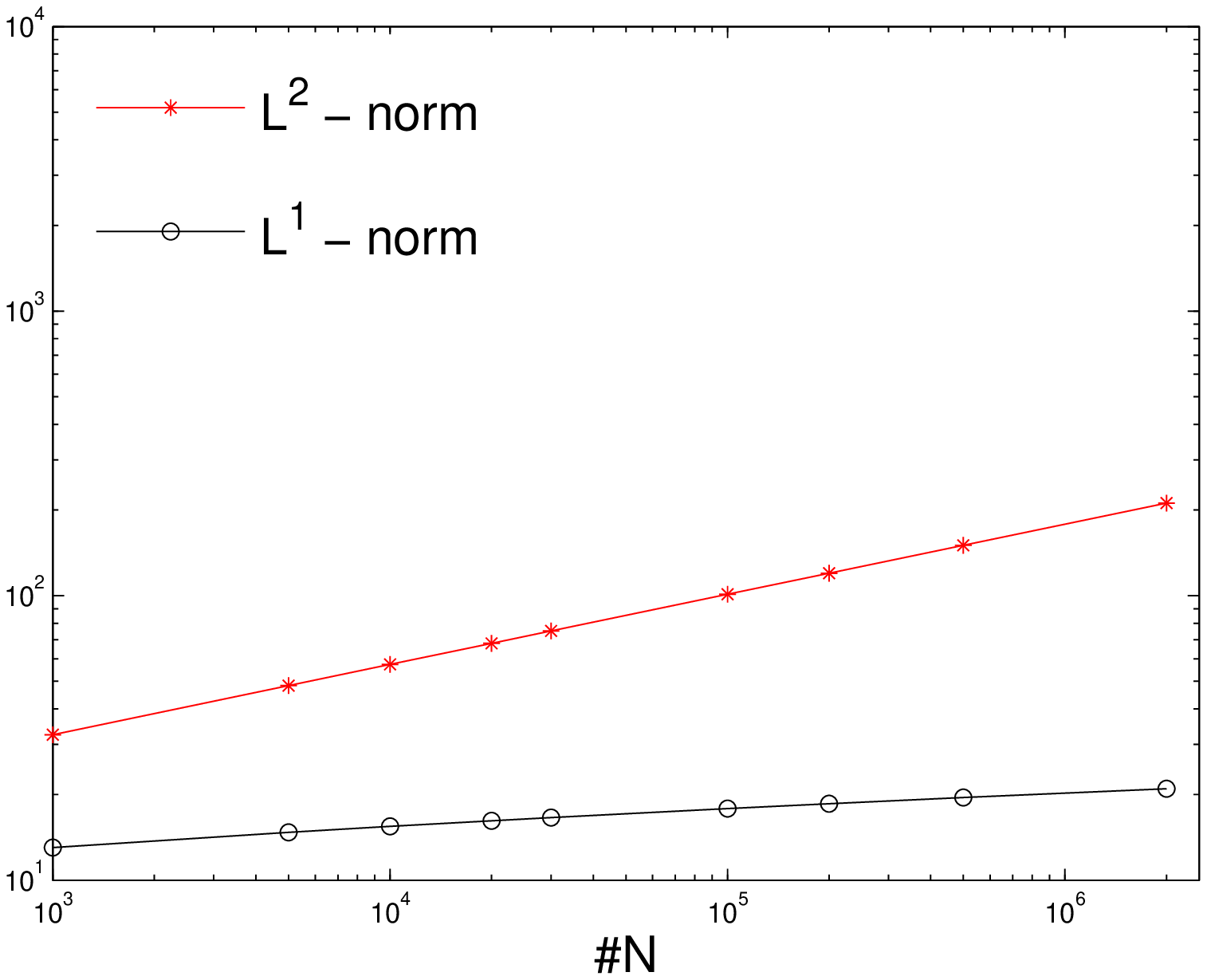}
	  \caption{Case \S~\ref{Exmpl2}}
	  \label{fig:case3}
        \end{subfigure}
        \begin{subfigure}[b]{0.49\textwidth}
	  \includegraphics[width=\textwidth]{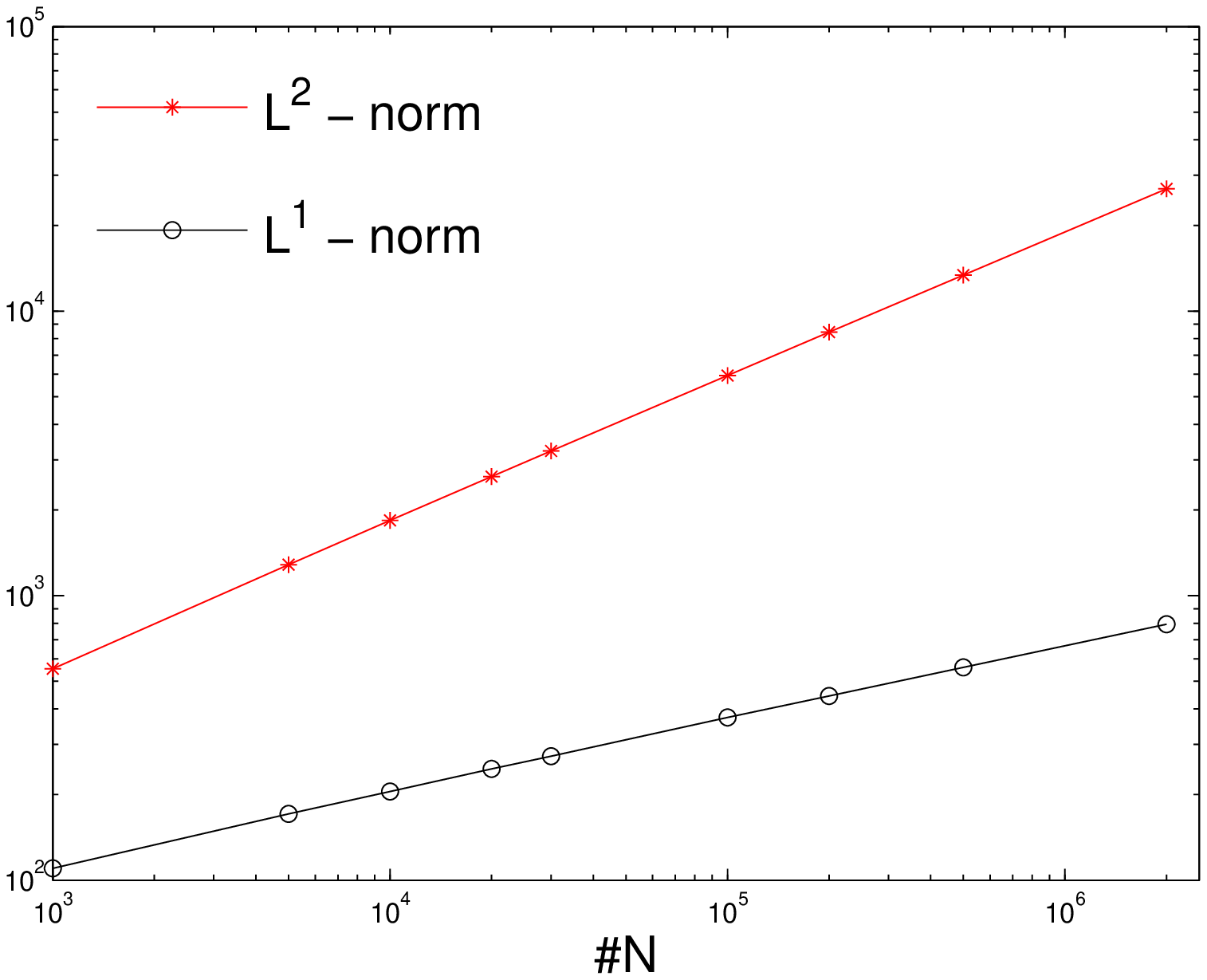}
	  \caption{Case \S~\ref{Exmpl3}}
	  \label{fig:case4.1}
        \end{subfigure}
        \begin{subfigure}[b]{0.49\textwidth}
	  \includegraphics[width=\textwidth]{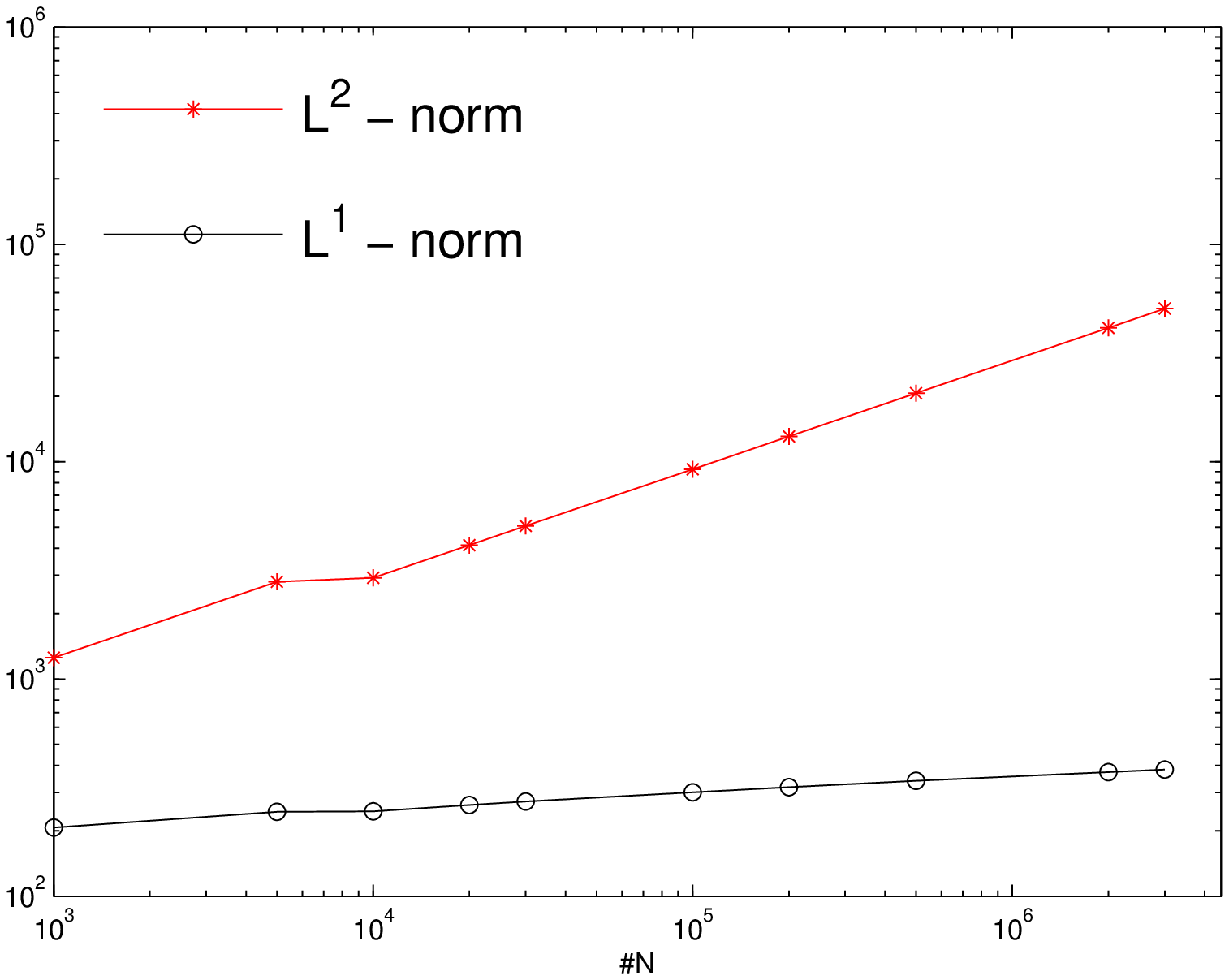}
	  \caption{Case \S~\ref{Exmpl4}}
	  \label{fig:case4.2}
        \end{subfigure}
\caption{Norm of the numerical solution}\label{fig:boundedness_of_norm}
\end{figure}

\subsection{Moments of the numerical solution}\label{subsec:numerics_moments}
We start by investigating the boundedness of the norm of the 
discrete solution obtained by means of \eqref{eq:discrete_form_q}, with $q=0$  
(which results in a time stepping which is the modification of the usual 
Crank-Nicolson scheme). The space $V_h$ is chosen to be a Lagrange 
finite element space of degree $1$, that is, $V_h$ is spanned by the usual 
``hat 
functions''.
We choose for simplicity $\Omega = [-0.5,0.5]$ and $\bP$ to be the 
Lebesgue measure and to simplify the computation, we treat $\omega$ 
as a parameter and $a(\omega)$ and $c_0(\omega)$ as deterministic functions of 
$\omega$, so that we 
can compute the $\norm{L^p(\Omega;\cdot)}{\cdot}$-norms by means of 
suitable quadrature rules rather than relying on expensive and slow Monte Carlo 
simulations. The function $g(t)$ is given by $\sin(\pi t)\sin(\pi \xi_1)\sin( 
\pi \xi_2)$, where 
$\xi$ denotes the spatial variable.
 The parameter $N$ against which we plot the 
norm of the discrete solution indicates the number of points used to 
approximate the integral with respect to $\omega$. The results are summarized 
in Figure~\ref{fig:boundedness_of_norm}.

\subsubsection{$\Am^{-1} \in L^{\infty}(\Omega) $, $\AM \notin 
L^{\infty}(\Omega) $ and $ f \in L^{\infty}(\Omega)$}\label{Exmpl1}
We fix $ a(\omega) = 1 + \frac{1}{\omega^2}$ and $c_0(\omega) = 
1+
\omega^3$. The operator presents now a non-integrable singularity which 
makes $\AM$ not uniformly bounded with respect to $\omega$. However, as 
expected from our theoretical results, we can see 
how the $L^p$-norm of the solution is finite, for different values of $p$, 
although $\AM \notin L^p$ for any $p$. This is visible in 
Figure~\ref{fig:case2}.

\subsubsection{$\Am^{-1} \notin L^{\infty}(\Omega) $, $\AM \in 
L^{\infty}(\Omega) $ and $ f \in L^{\infty}(\Omega)$}
\label{Exmpl2}
We fix now $ a(\omega) = \abs{\omega}^{\alpha}$ and $c_0(\omega) = 1 + 
\omega^3$. 
The operator is now not uniformly coercive with respect to $\omega$. We choose 
in particular $\alpha =0.99$, so that we expect the solution to have finite 
mean but infinite mean-square. We can appreciate the numerical results in 
Figure~\ref{fig:case3}, where we can see how the $L^2(\Omega;\cY)$-norm of the 
solution diverges as $N$ grows. On the other hand we can see how the 
$L^1(\Omega;\cY)$-norm seems to approach a finite value, although slowly, due 
both to the choice of $\alpha=0.99$, and to the quadrature rule used.

\subsubsection{$\Am^{-1} \notin L^{\infty}(\Omega) $, $\AM \in 
L^{\infty}(\Omega) $ and $ f \notin L^{\infty}(\Omega)$}
\label{Exmpl3}
We fix now $ a(\omega) = {\abs{\omega}^{\alpha}}$ and $c_0(\omega) = 
\frac{1}{\abs{\omega}^{\beta}}$.
The operator is now not uniformly coercive with respect to $\omega$ and the 
right-hand side is not uniformly bounded. We choose 
 $\alpha =0.99$ and $\beta =0.5$ and, as expected, the solution does not have 
finite 
mean nor finite mean-square. This is visible in Figure~\ref{fig:case4.1}.

\subsubsection{$\Am^{-1} \notin L^{\infty}(\Omega) $, $\AM \in 
L^{\infty}(\Omega) $ and $ f \notin L^{\infty}(\Omega)$, different 
singularities}
\label{Exmpl4}
By changing the location of the 
singularity of $f$ in \S~\ref{Exmpl3} , for example by having $c_0(\omega) = 
\frac{1}{\abs{\omega-0.4}^{\beta}}$, we obtain a solution which still has 
finite 
mean, although the hypotheses of Theorem~\ref{thm:Lptheorem2} are not 
fulfilled (this fact reflects what is pointed out in 
Example~\ref{example:factorization_op}). This is visible in 
Figure~\ref{fig:case4.2}, 
where we can see how the $L^1(\Omega;\cY)$-norm appears to slowly converge to a 
finite value, while the $L^2(\Omega;\cY)$-norm diverges.

\subsection{Convergence in $L^p$}\label{subsec:numerics_Lp-convergence}
In this section we show how the error converges not only $\omega$-wise, but 
also in an $L^p$-sense, whenever the solution is bounded with respect to the 
same $L^p$-norm. We restrict 
ourselves to the case of finite moments and present the mean error
of the Petrov-Galerkin discretization.

\begin{figure}[h!]
        \centering
        \begin{subfigure}[b]{0.49\textwidth}
	  \includegraphics[width=\textwidth]{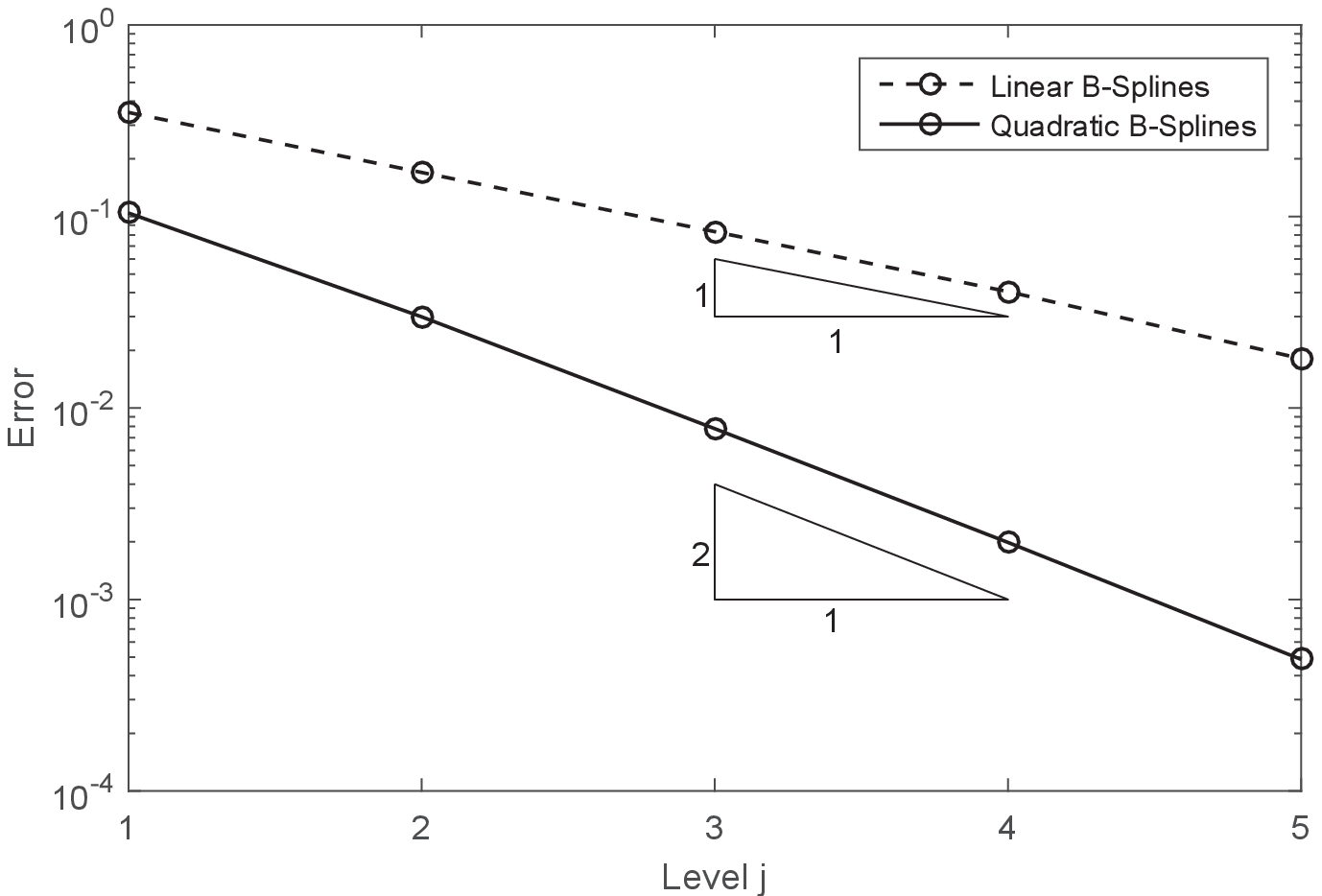}
	  \caption{ $a(\omega):=\abs{\omega}^{0.99}$ 
	  and $c_0:= {\abs{\omega-0.4}}^{-\frac12}$,\\ linear and 
quadratic B-splines in space }
	  \label{fig:Exmpl6_discr}
        \end{subfigure}
                \begin{subfigure}[b]{0.49\textwidth}
	  \includegraphics[width=\textwidth]{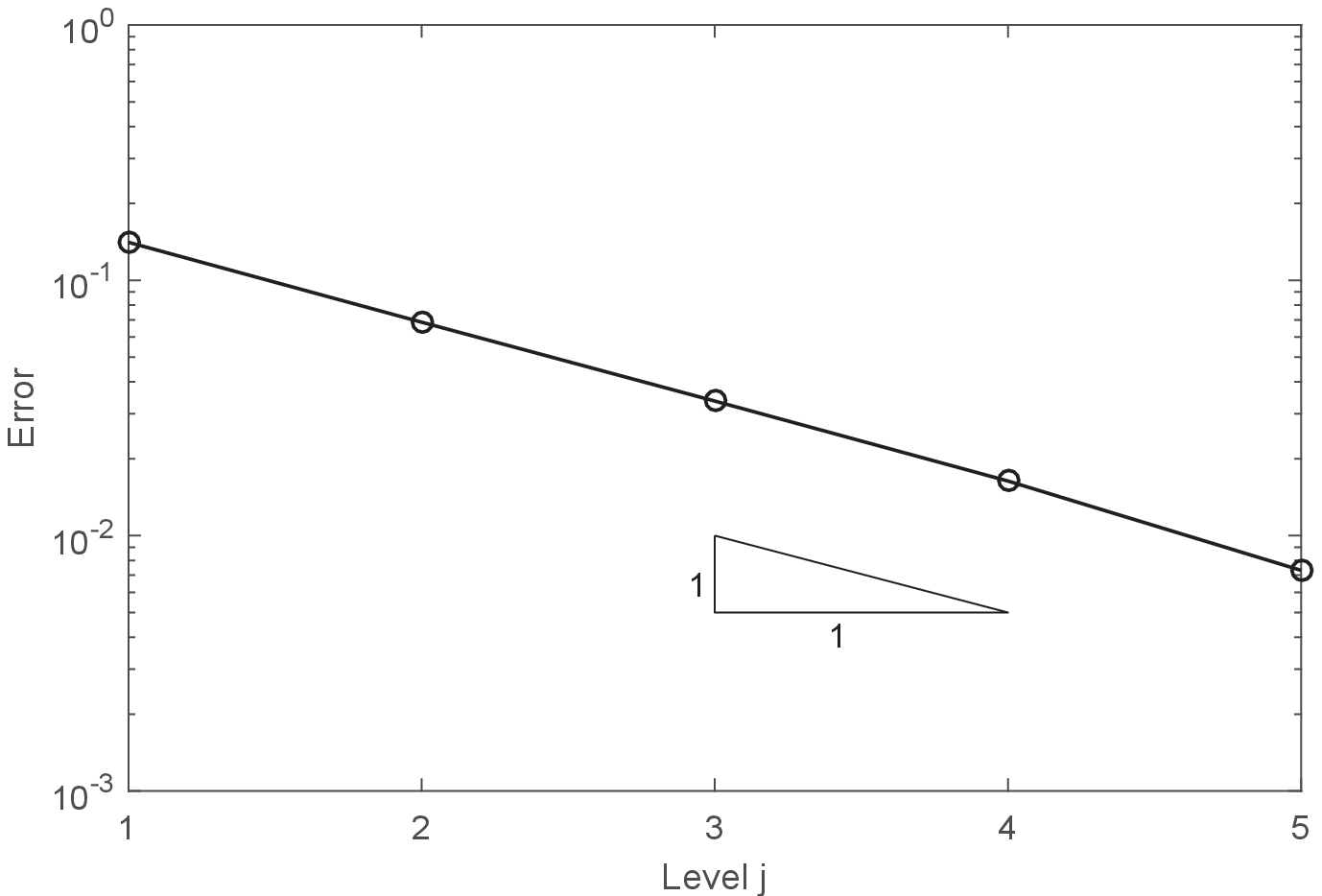}
	  \caption{ $a \sim 
\mathcal{LN}(0,1)$ and $c_0\equiv 1$,\\ linear B-splines in space }
 	  \label{fig:Exmpl5_discr}
        \end{subfigure}
\caption{Convergence rate for $\mathbb{E} 
\|u-U\|_{\cY}$}\label{fig:convergence}
\end{figure}
To this end we consider \eqref{eq:randHeatStrong_numerics} again on 
temporal domain
$[0,1]$ and spatial domain $\Lambda:=[0,1]$ and show the optimal convergence 
for the counterpart of \S~\ref{Exmpl4}, which is the sharpest example. 
This is visible in Figure~\ref{fig:Exmpl6_discr}.

For the discretization we choose $S_{k,1}$ and $Q_{k,0}$ as discrete temporal 
trial and test spaces on a uniform grid
and $V_h$ as the space of continuous and piecewise linear functions on a 
uniform spatial grid of grid size $h$. In Figure \ref{fig:Exmpl6_discr}
we also consider $V_h$ spanned by quadratic B-splines, i.e.,
piecewise quadratic and globally differentiable functions.
In view of the CFL-condition we set $h(j):=2^{-j}$ and $k(j):=2^{-2j}$.
We can see that the error converges with optimal order
with respect to the $L^1(\Omega;\mathcal{Y})$-norm, i.e., the mean of the 
errors in energy norm. 

In addition to the examples above, we consider an example with non-uniformly 
distributed coefficients.
To this end, we set $a \sim \mathcal{LN}(0,1)$ log-normally distributed and 
$c_0\equiv 1$.
A log-normally distributed $a(\omega)$ is neither uniformly bounded nor 
uniformly coercive, but it is
not hard to see that the moments of $a(\omega)$ and $\frac{1}{a(\omega)}$ 
increase but are finite for 
arbitrary finite order of moments. The results are illustrated in Figure 
\ref{fig:Exmpl5_discr}. As expected we can see
that we achieve the optimal order also in this case.

\section{Final remarks}
In this paper we developed a quasi-optimality theory for the solution of 
parabolic problems with random coefficients, under relatively weak 
assumptions of non-uniform bounds on the operator $A$ that defines the 
equation. All the bounds that we obtain are explicit, and their 
$\omega$-dependence is sharply tracked. Despite the fact that we do not have 
theoretical proof of their sharpness, numerical experiments suggest that the 
bounds for the norm of the continuous, discrete and semidiscrete solution are 
as 
sharp as possible. Furthermore, the results of quasi-optimality for spatial 
discretizations of the problem are obtained in terms of an absolute constant, 
which depends on $\omega$ only through $\rho = \frac{\AM}{\Am}$, while for full 
discretizations the quasi-optimality constant contains an additional factor 
proportional to $\AM$. Although 
this extra factor seems unavoidable, if we follow our approach, we cannot claim 
that what we achieve is the best possible bound. In particular, a 
combination of our results for spatial semidiscretization together with a 
suitable time-stepping, might lead to optimal results of convergence in 
$L^p(\Omega;\cdot)$, and this kind of analysis could be the next step for 
future research in this field.

\bibliographystyle{alpha}
\bibliography{biblionew}

\def\cprime{$'$}
\begin{thebibliography}{LM16b}

\bibitem[And12]{AndreevPhd}
R.~Andreev.
\newblock {\em Stability of space-time {P}etrov-{G}alerkin discretizations for
  parabolic evolution problems}.
\newblock PhD thesis, ETH Z{\"u}rich, 2012.

\bibitem[And13]{Andreev1}
R.~Andreev.
\newblock Stability of sparse space-time finite element discretizations of
  linear parabolic evolution equations.
\newblock {\em IMA J. Numer. Anal.}, 33(1):242--260, 2013.

\bibitem[BJ89]{BabuskaJanik}
I.~Babu{\v{s}}ka and T.~Janik.
\newblock The {$h$}-{$p$} version of the finite element method for parabolic
  equations. {I}. {T}he {$p$}-version in time.
\newblock {\em Numer. Methods Partial Differential Equations}, 5(4):363--399,
  1989.

\bibitem[BY14]{Bank}
E.~Bank and H.~Yserentant.
\newblock On the {$H^1$}-stability of the {$L_2$}-projection onto finite
  element spaces.
\newblock {\em Numer. Math.}, 126(2):361--381, 2014.

\bibitem[Cha12]{C12}
J.~Charrier.
\newblock Strong and weak error estimates for elliptic partial differential
  equations with random coefficients.
\newblock {\em Siam J. Numer. Anal.}, 50(1):216--246, 2012.

\bibitem[CS11]{CheginiStev}
N.~Chegini and R.~Stevenson.
\newblock Adaptive wavelet schemes for parabolic problems: sparse matrices and
  numerical results.
\newblock {\em SIAM J. Numer. Anal.}, 49(1):182--212, 2011.

\bibitem[DL92]{DauLio}
R.~Dautray and J.L. Lions.
\newblock {\em Mathematical {A}nalysis and {N}umerical {M}ethods for {S}cience
  and {T}echnology. {V}ol. 5}.
\newblock Springer-Verlag, Berlin, 1992.
\newblock Evolution problems. I, With the collaboration of Michel Artola,
  Michel Cessenat and H{\'e}l{\`e}ne Lanchon, Translated from the French by
  Alan Craig.

\bibitem[GAS14]{SchwabAndreev}
C.~J. Gittelson, R.~Andreev, and C.~Schwab.
\newblock {O}ptimality of adaptive {G}alerkin methods for random parabolic
  partial differential equations.
\newblock {\em Journal of computational and applied mathematics}, 263:189--201,
  2014.

\bibitem[Hac81]{Hackbusch}
W.~Hackbusch.
\newblock Optimal {$H^{p,\,p/2}$} error estimates for a parabolic {G}alerkin
  method.
\newblock {\em SIAM J. Numer. Anal.}, 18(4):681--692, 1981.

\bibitem[LM16a]{LarssonMolteni2}
S.~Larsson and M.~Molteni.
\newblock Numerical solution of parabolic problems based on a weak space-time
  formulation.
\newblock Preprint, arXiv:1603.03210, 2016.

\bibitem[LM16b]{LarssonMolteni}
S.~Larsson and M.~Molteni.
\newblock A weak space-time formulation for the linear stochastic heat
  equation.
\newblock {\em Int. J. Appl. Comput. Math.}, 2016.
\newblock Electronic.

\bibitem[Mol13]{Mollet}
C.~Mollet.
\newblock Stability of {P}etrov-{G}alerkin discretizations: {A}pplication to
  the space-time weak formulation for parabolic evolution problems.
\newblock {\em Comput. Methods. Appl. Math.}, 14(2):231--255, 2013.

\bibitem[Mol16]{Mo_Diss}
C.~Mollet.
\newblock {\em Parabolic PDEs in Space-Time Formulations: Stability for
  Petrov-Galerkin-Formulations with B-Splines and Existence of Moments for
  Problems with Random Coefficients}.
\newblock PhD thesis, University of Cologne, 2016.
\newblock in preparation.

\bibitem[NSV09]{NSV}
R.~H. Nochetto, K.~G. Siebert, and A.~Veeser.
\newblock Theory of adaptive finite element methods: An introduction.
\newblock In R.~DeVore and A.~Kunoth, editors, {\em Multiscale, Nonlinear and
  Adaptive Approximation}, pages 409--542. Springer Berlin Heidelberg, 2009.

\bibitem[SS09]{SchwabSte}
C.~Schwab and R.~Stevenson.
\newblock Space-time adaptive wavelet methods for parabolic evolution problems.
\newblock {\em Math. Comp.}, 78(267):1293--1318, 2009.

\bibitem[SS11]{SchwabSuli}
C.~Schwab and E.~S{\"u}li.
\newblock Adaptive {G}alerkin approximation algorithms for partial differential
  equations in infinite dimensions.
\newblock Numerical Analysis Technical Report No. 1452, University of Oxford,
  2011.

\bibitem[Tan13]{Fra}
F.~Tantardini.
\newblock Quasi {O}ptimality in the {B}ackward {E}uler-{G}alerkin {M}ethod for
  {L}inear {P}arabolic {P}roblems.
\newblock Tesi di dottorato, Universita' degli Studi di Milano, 2013.

\bibitem[Tec13]{Aretha}
A.~L. Teckentrup.
\newblock Multilevel {M}onte {C}arlo {M}ethods and {U}ncertainty
  {Q}uantification.
\newblock {P}h.{D}. thesis, {U}niversity of {B}ath, {D}epartment of
  {M}athematical {S}ciences, 2013.

\bibitem[UP12]{UrbanPatera}
K.~Urban and A.~T. Patera.
\newblock A new error bound for reduced basis approximation of parabolic
  partial differential equations.
\newblock {\em C. R. Math. Acad. Sci. Paris}, 350(3-4):203--207, 2012.

\bibitem[UP14]{UrbanPatera2}
K.~Urban and A.~T. Patera.
\newblock An improved error bound for reduced basis approximation of linear
  parabolic problems.
\newblock {\em Math. Comp.}, 83(288):1599--1615, 2014.

\bibitem[XZ03]{XuZik}
J.~Xu and L.~Zikatanov.
\newblock Some observations on {B}abu\v ska and {B}rezzi theories.
\newblock {\em Numer. Math.}, 94(1):195--202, 2003.

\end{thebibliography}

\end{document}